\documentclass[reqno]{amsart}
\usepackage{amsfonts,amsmath,amssymb,color}

\numberwithin{equation}{section}

\usepackage[babel=true,kerning=true]{microtype}

\usepackage{amsthm,leqno}

\usepackage{amsfonts,amsmath,amssymb,enumerate}
\usepackage{graphicx}
\usepackage{tikz}
\usetikzlibrary{decorations.markings}
\usetikzlibrary{plotmarks}
\usetikzlibrary{patterns}

\newtheorem{thm}{Theorem}

\newtheorem{prop}{Proposition}[section]

\newtheorem{claim}{Claim}[section]

\newcommand{\be}{\begin{equation*}}
\newcommand{\ee}{\end{equation*}}
\newcommand{\beq}{\begin{equation}}
\newcommand{\eeq}{\end{equation}}
\newcommand{\begincal}{\begin{eqnarray*}}
\newcommand{\fincal}{\end{eqnarray*}}
\newcommand{\ds}{\displaystyle}

\newcommand{\eps}{\varepsilon}

\newcommand{\xia}{x_{i,\alpha}}
\newcommand{\xja}{x_{j,\alpha}}
\newcommand{\xa}{x_\alpha}
\newcommand{\tmua}{\tilde{\mu}_\alpha}
\newcommand{\ma}{\mu_\alpha}
\newcommand{\tua}{\tilde{u}_\alpha}
\newcommand{\tha}{\tilde{h}_\alpha}
\newcommand{\tfa}{\tilde{f}_\alpha}
\newcommand{\taa}{\tilde{a}_\alpha}
\newcommand{\mua}{\mu_\alpha}
\newcommand{\hha}{\hat{h}_\alpha}
\newcommand{\hfa}{\hat{f}_\alpha}
\newcommand{\haa}{\hat{a}_\alpha}
\newcommand{\hWa}{\hat{W}_\alpha}

\newcommand{\hXa}{\hat{X}_\alpha}
\newcommand{\hYa}{\hat{Y}_\alpha}

\newcommand{\tya}{\tilde{y}_\alpha}
\newcommand{\va}{v_\alpha}
\newcommand{\hpa}{\hat{\rho}_\alpha}
\newcommand{\tva}{\tilde{v}_\alpha}
\newcommand{\cva}{\check{v}_\alpha}
\newcommand{\cWa}{\check{W}_\alpha}
\newcommand{\cVa}{\check{V}_\alpha}
\newcommand{\cZa}{\check{Z}_\alpha}
\newcommand{\cXa}{\check{X}_\alpha}
\newcommand{\cYa}{\check{Y}_\alpha}
\newcommand{\cha}{\check{h}_\alpha}
\newcommand{\cfa}{\check{f}_\alpha}
\newcommand{\caa}{\check{a}_\alpha}
\newcommand{\cv}{\check{v}}
\newcommand{\tWa}{\tilde{W}_\alpha}
\newcommand{\Dx}{\overrightarrow{\Delta}_\xi}
\newcommand{\Dh}{\overrightarrow{\Delta}_h}
\newcommand{\Sp}{{\mathbb S}^3}
\newcommand{\RR}{{\mathbb R}}

\begin{document}
\title[The Einstein-Lichnerowicz constraints system]
{Stability of the Einstein-Lichnerowicz constraints system} 

\author{Olivier Druet } 
\address{Olivier Druet, CNRS UMR 5208, Institut Camille Jordan, Universit\'e de Lyon 1, France.}
\email{druet@math.univ-lyon1.fr}

\author{Bruno Premoselli } 
\address{Bruno Premoselli, Universit\'e de Cergy-Pontoise 
D\'epartement de Math\'ematiques 
2 avenue Adolphe Chauvin 
95302 Cergy-Pontoise cedex 
France}
\email{bruno.premoselli@u-cergy.fr}

\date{May 2014}

\begin{abstract} 
We study the Einstein-Lichnerowicz constraints system,
obtained through the conformal method when addressing the initial data
problem for the Einstein equations in a scalar field theory. We prove that this
system is stable with respect to the physics data when posed on the 
standard $3$-sphere. 
\end{abstract}

\maketitle 

\section{Introduction}

The Einstein field equations are a set of 10 equations in the theory of relativity which relate the geometry of a space-time to the distribution of matter and energy. Given a space-time $(\mathcal{M},\tilde g)$ -- namely a $(3+1)$-dimensional Lorentzian manifold -- they are written as
\begin{equation} \label{ee} \tag{E}
\textrm{Ric}(\tilde g)_{ab} - \frac{1}{2} R(\tilde g) \tilde g_{ab} = T_{ab},
\end{equation}
where $\mathrm{Ric}(\tilde g)$ is the Ricci tensor of $\tilde g$, $R(\tilde g)$ is the scalar curvature of $\tilde g$ and $T$ is the stress-energy tensor which depends on the model used to represent the distribution of matter and energy. In the scalar field setting, the stress-energy tensor takes the form:
\begin{equation} \label{T}
T_{ab} = \nabla_a \psi \nabla_b \psi - \left(\frac{1}{2} |\nabla \psi |_{\tilde g}^2 + V(\psi) \right) {\tilde g}_{ab},
\end{equation}
where the scalar-field $\psi$ is a smooth function in $\mathcal{M}$, $V$ - the potential associated to $\psi$ - is a smooth function in $\RR$ and $\psi$ and $V$ satisfy:
\begin{equation} \label{divnulle} 
 \nabla^i \nabla_i \psi = \frac{dV}{d \psi},
 \end{equation}
where $\nabla^i \nabla_i \psi$ is the laplacian for the Lorentzian metric $\tilde g$, so that there holds $\nabla_i T^{ij} = 0$. 
When looking for solutions of \eqref{ee}, in order to avoid pathological examples from a physical point of view, the theory restricts itself to a class of space-times where the Einstein equations can be seen as an evolution problem from a given initial data. Following standard terminology, 
an initial data for \eqref{ee} consists of $(M,h,K,\psi_0,\psi_1)$ where $(M,h)$ is a $3$-dimensional Riemannian manifold, $K$ is a smooth $(2,0)$-tensor and $\psi_0$ and $\psi_1$ are smooth functions in $M$ satisfying the following \emph{constraint equations} in $M$:
\begin{equation} \label{c}  
\left \{
\begin{aligned}
& R(h) + (\mathrm{tr}_h K)^2 - || K ||_h^2 = \psi_1^2 + |\nabla \psi_0|_h^2 + 2 V(\psi_0) \\
& \nabla^j K_{ij} - \nabla_i (\mathrm{tr}_h K)  = \psi_1 \nabla_i \psi_0 .
\end{aligned}
\right.
\end{equation}
Then, an initial data set $(M,h,K,\psi_0,\psi_1)$ is said to admit a globally hyperbolic space-time development if there exists a Lorentzian manifold $(\mathcal{M},\tilde g)$, a smooth function $\psi$ in $\mathcal{M}$ and an embedding $i : M \to \mathcal{M}$ such that \eqref{ee} and \eqref{divnulle} are satisfied, such that $i(M)$ is a Cauchy hypersurface 
of $\mathcal{M}$ and such that $i^{*} \tilde g = h$, $i^* \mathcal{K} = K$, $\psi \circ i = \psi_0$ and $(N \cdot \psi )\circ i = \psi_1$, where $\mathcal{K}$ is the second fundamental form of $i(M)$ in $\mathcal{M}$ and $N$ is the future directed timelike unit normal to $i(M)$. Straightforward application of the Gauss and Codazzi equation shows that a necessary condition for a globally hyperbolic space-time $(\mathcal{M}, \tilde g)$ with a scalar-field $\psi$ to solve \eqref{ee} and \eqref{divnulle} is that the system \eqref{c} be satisfied on a Cauchy hypersurface of $\mathcal{M}$. Since the works of Choquet-Bruhat \cite{ChoBru} and Choquet-Bruhat and Geroch \cite{ChoGe}, it is known that \eqref{c} is also a sufficient condition on an initial data set $(M, h, K, \psi_0, \psi_1)$ to admit a maximal globally hyperbolic development (MGHD). This object provides a framework for the main open conjectures in general relativity and for the analysis of solutions of \eqref{ee}. 

\medskip The constraint equations \eqref{c} are a highly underdetermined system of $4$ equations for $14$ unknowns. To overcome this problem, the conformal method, initiated by Lichnerowicz \cite{Lich}, freezes some of the unknown variables -- considering them from now on as parameters -- and proposes to solve the system for the remaining set of variables. In 
this process, one looks for the unknown metric $h$ in the conformal class of a reference Riemannian metric $g$ in $M$. 
We consider the case of closed manifolds in what follows, where closed means compact without boundary. We let $(M,g)$ be a closed Riemannian 
$3$-manifold. By the conformal method, in order to obtain a solution of \eqref{c} in $M$ it is enough to solve the following elliptic system of equations, where the unknown $(\varphi,W)$ are a smooth positive function on $M$ and a smooth $1$-form on $M$:
\begin{equation} \label{cc} \tag{$C_F$} 
\left \{ 
\begin{aligned} 
&    \triangle_g \varphi + \mathcal{R}_\psi  \varphi  = \mathcal{B}_{\tau, \psi, V} \varphi^{5} + \frac{\mathcal{A}_{\pi, U}(W)}{ \varphi^{7}}~,  \\ 
 & \overrightarrow{\triangle_{g}} W  = - \frac{2}{3}\varphi^{6} \nabla\tau - \pi\nabla \psi~,   
\end{aligned}
\right.
\end{equation}
where
\begin{equation} \label{expressions}
\begin{aligned}
& \mathcal{R}_{\psi}  = \frac{1}{8} \left( R(g) - |\nabla \psi|_g^2 \right)~~,~~\mathcal{B}_{\tau,\psi,V}  = \frac{1}{8} \left( 2 V(\psi) - \frac{2}{3} \tau^2 \right), \\
& \mathcal{A}_{\pi,U}(W)  = \frac{1}{8} \left(  |U + \mathcal{L}_g W |_g^2 + \pi^2 \right)~~,\\
\end{aligned} \end{equation}
$R(g)$ is the scalar curvature of $(M,g)$ and $\mathcal{L}_g W$ -- the conformal Killing operator -- is a $(2,0)$-tensor defined in coordinates by:
\begin{equation} \label{conflie}
  \mathcal{L}_gW_{ij} = W_{i,j} + W_{j,i} - \frac{2}{3} \left(\mathrm{div}_g W\right) g_{ij}\hskip.1cm. 
\end{equation}
Also, in \eqref{cc}, $\triangle_g = - \mathrm{div}_g (\nabla \cdot)$ is the Laplace-Beltrami operator and we have let:
\[\overrightarrow{\triangle_{g}} W = - \mathrm{div}_g (\mathcal{L}_g W)\]
for any $1$-form $W$. Note that in \eqref{cc} $6$ is the critical exponent for the embedding of the Sobolev space $H^1(M)$ into Lebesgue spaces. The initial data in \eqref{cc} are $F = (\tau, \psi, \pi, U)$, the unknowns are $\varphi$ and $W$. 
We let $\mathcal{F}$ be the initial data set
\begin{equation} \label{initialdata}
 \mathcal{F} = \left\{(\tau, \psi, \pi, U)~,~\tau, \psi, \pi \in C^\infty(M), U \in T_{(2,0)}(M)\right\},
 \end{equation}
where $T_{(2,0)}(M)$ denotes the set of smooth symmetric traceless and divergence-free $(2,0)$-tensors in $M$. We endow $\mathcal{F}$ with the following norm: 
for $F = (\tau, \psi, \pi, U )\in \mathcal{F}$,
\begin{equation} \label{Fnorm}
 \Vert F \Vert_\mathcal{F} = \Vert \tau \Vert_{C^2(M)} + \Vert \psi \Vert_{C^1(M)} + \Vert \pi \Vert_{C^0(M)} + \Vert U \Vert_{C^0(M)} .
 \end{equation}
Given an initial data $F = (\tau, \psi, \pi, U)$ and a solution $(\varphi, W)$ of \eqref{cc} one obtains an initial data set $(M,h,K,\psi_0,\psi_1)$ solution of the original constraint equations \eqref{c} by letting
\begin{equation} \label{dataparam}
(h,K,\psi_0, \psi_1) = (\varphi^4g, \frac{\tau}{3}\varphi^4g + \varphi^{-2}(U + \mathcal{L}_g W), \psi, \varphi^{- 6} \pi). 
\end{equation}
Conversely, starting from a solution of \eqref{c} one obtains a solution $(\varphi, W)$ of \eqref{cc} for suitably defined parameters $F = (\tau, \psi, \pi, U)$ according to \eqref{dataparam}. Note that then $\tau$ turns out to be the mean curvature of the embedding of $M$ in its globally hyperbolic development.

\medskip Different physical setting are roughly distinguished by the sign of the coefficient $\mathcal{B}_{\tau, \psi, V}$. In the vacuum case, $\mathcal{B}_{\tau, \psi, V}$ is nonpositive while it becomes positive if we allow a cosmological constant $\Lambda$ (corresponding to the case $V(\psi) \equiv \Lambda$) and $\tau$ is not too big with respect to this constant. The $\mathcal{B}_{\tau, \psi, V}$ positive case has received increasing attention in recent years with the attempts to use scalar-field theories to model the observed acceleration of the expansion of the universe, see Rendall \cite{Rendall}. Several existence results for \eqref{cc} are known. When $M$ is closed these are of different nature according to the sign of $\mathcal{B}_{\tau, \psi, V}$ and the case $\mathcal{B}_{\tau, \psi, V} > 0$ turns out to exhibit a rich behavior where solutions are not necessarily unique as shown in Premoselli \cite{Premoselli2} or Holst - Meier \cite{HolstMeier}. If $\nabla \tau = 0$ (in which case the system \eqref{cc} is decoupled), existence results are in Isenberg \cite{Ise} for the $\mathcal{B}_{\tau, \psi, V} \le 0$ case and in Hebey-Pacard-Pollack \cite{HePaPo} for the $\max_M \mathcal{B}_{\tau, \psi, V} > 0$ case. If $\tau$ is not constant, existence results are in Dahl-Gicquaud-Humbert \cite{DaGiHu} and the references therein if $\mathcal{B}_{\tau, \psi, V} \le 0$ and Premoselli \cite{Premoselli1} if $\mathcal{B}_{\tau, \psi, V} > 0$. 
As \eqref{dataparam} shows, solving the system \eqref{cc} provides solutions of \eqref{c} parameterized by some initial data in $\mathcal{F}$. We are in this work interested in the stability of the system \eqref{cc} with respect to perturbations both of the initial data in 
$\mathcal{F}$ and of the potentials $V$ in $C^1\left(\RR\right)$. We fully answer the question in the positive case 
of the round $3$-sphere.

\begin{thm} \label{Th1}
Let $(M,g) = (\mathbb{S}^3, h)$ be the standard round $3$-sphere.  Let $F_0 = (\tau, \psi, \pi, U) \in \mathcal{F}$ and 
$V_0 \in C^1\left(\mathbb R\right)$ be such that $\mathcal{B}_{\tau, \psi, V_0} > 0$ in $\mathbb{S}^3$, $\triangle_{h} +  \mathcal{R}_{\psi}$ is coercive and 
$\pi \not \equiv 0$. For any sequence $\left(F_\alpha\right)$, $F_\alpha = \left(\tau_\alpha, \psi_\alpha, \pi_\alpha, U_\alpha\right)$, 
of initial data in 
$\mathcal{F}$, any sequence $(V_\alpha)$ 
of potentials in $C^1\left(\RR\right)$, such that
$$\Vert F_\alpha - F_0 \Vert_{\mathcal{F}} + \Vert V_\alpha-V \Vert_{C^1\left(\RR\right)} \to 0$$
as $\alpha \to \infty$, where $\Vert \cdot \Vert_\mathcal{F}$ is as in \eqref{Fnorm}, and any sequence 
$\left(\varphi_\alpha, W_\alpha\right)$ of solutions of the conformal constraint 
system \eqref{cc} with $F = F_\alpha$ and $V = V_\alpha$, there exists a solution $(\varphi_0, W_0)$ of the limit system \eqref{cc} with $F = F_0$ 
and $V = V_0$ such that, up to a subsequence, $\left(\varphi_\alpha, W_\alpha\right)$ converges to $(\varphi_0, W_0)$ in $C^{1, \theta}(M)$ as $\alpha\to +\infty$, where $0 < \theta < 1$ is arbitrary.
\end{thm}

Theorem \ref{Th1} is of course a compactness result. Sequences of solutions of small perturbations of the constraint system \eqref{cc} with respect to the initial data 
and the potential converge smoothly to a solution of the unperturbed system. In particular, 
small perturbations of the initial data do not create sequences of solutions far away from the set of solutions of the original system. 

\medskip When combined with local well-posedness results for quasilinear hyperbolic systems, Theorem \ref{Th1} provides a stability result on spacetime developments with respect to the physics data in small time in the sense that the spacetime development corresponding to a perturbation of the initial scalar-field data $( \tau, \psi, \pi, U)$ will be close in strong sense, for small times, to a spacetime development of the original 
problem. In particular small numerical errors in $(\tau, \psi, \pi, U)$ do not affect the solutions of \eqref{ee} for small time.

\medskip The regularity of convergence of solutions in Theorem \ref{Th1} depends on the convergence of the initial data. We stated it under reasonable assumptions on the convergence of the $F_\alpha$'s and $V_\alpha$'s,  but due to its elliptic features, the system \eqref{cc} is regularizing: if the perturbed initial data as well as the potentials converge in $C^{k, \theta}$, $k \ge 2$, $0 < \theta <1$, then the solutions converge in $C^{k+1,\theta'}(M)$ for all $\theta' < \theta$.

\medskip
We state Theorem \ref{Th1} in the model case of $\mathbb{S}^3$ in order to slightly simplify the presentation of the proof but our arguments easily extend to the general case of a $3$-dimensional locally conformally flat manifold. Also note that the only relevant assumptions in Theorem \ref{Th1} are $\mathcal{B}_{\tau, \psi, V_0} > 0$ and 
$\pi \not \equiv 0$. Assuming that $\mathcal{B}_{\tau, \psi, V_0} > 0$, 
the coercivity of $\triangle_{g_s} + \mathcal{R}_{\psi}$ is a necessary condition for the existence of solutions of the scalar equation of \eqref{cc}. 

\medskip The proof of Theorem \ref{Th1} goes through the proof of an involved stability result concerning a slightly wider class of systems than \eqref{cc}, see \eqref{equationEalpha} in section \ref{sectionstatement}  below. Section \ref{sectionproof} contains the arguments in the proof of Theorem \ref{Th1} and uses the pointwise description of sequences of blowing-up solutions of \eqref{cc} around a concentration point. Such a pointwise description is obtained in section \ref{sectionlba} and is the core of the analysis of the paper. It requires a simultaneous investigation of the defects of compactness that may occur in each of the two equations of \eqref{cc}. Finally, sections 6 to 9 gather some technical results used throughout the paper.

\medskip {\bf Acknowledgements} - The two authors wish to thank Emmanuel Hebey for fruitful discussions during the preparation of this paper.

\section{A PDE result}\label{sectionstatement}

We shall in fact prove a more general result than Theorem \ref{Th1}. We consider a sequence $\left(u_\alpha,W_\alpha\right)$ of solutions on $\left(S^3,h\right)$ of 
\begin{equation}\label{equationEalpha}
\left\{\begin{array}{l}
\,\\{\ds \Delta_h u_\alpha + h_\alpha u_\alpha = f_\alpha u_\alpha^5 + \frac{a_\alpha}{u_\alpha^7}}\\
\,\\
{\ds \overrightarrow{\Delta_h} W_\alpha = u_\alpha^{6}X_\alpha +Y_\alpha}\\
\,\\
\end{array}\right. 
\end{equation}
where 
$$a_\alpha = \left\vert U_\alpha + {\mathcal L}_h W_\alpha\right\vert^2 + b_\alpha$$
and $\left(h_\alpha, f_\alpha, b_\alpha\right)$ are smooth functions with $f_\alpha>0$, $b_\alpha>0$, $X_\alpha$ and $Y_\alpha$ are smooth $1$-forms and $U_\alpha$ is a smooth traceless divergence free $(2,0)$-tensor field. We assume that 
\begin{equation}\label{eqconv}
\left(h_\alpha, f_\alpha, X_\alpha, Y_\alpha, U_\alpha, b_\alpha\right)\to \left(h_0, f_0, X_0, Y_0, U_0, b_0\right)\hbox{ as }\alpha\to +\infty
\end{equation}
where all the convergences take place in $C^0$, except the convergences of $X_\alpha$ and $f_\alpha$ to $X_0$ and $f_0$ which take place in $C^1$ with $f_0>0$. We assume moreover that $h_0$ is such that $\Delta_h + h_0$ is coercive and that 
\begin{equation}\label{eqnonzero}
b_0\not\equiv 0\hskip.1cm.
\end{equation}
In the following, we may assume that $W_\alpha$ is orthogonal to the set of Killing $1$-forms on $S^3$ (see proposition \ref{Killingsphere} for a classification of these Killing $1$-forms). Indeed, the only quantity entering into the equations is ${\mathcal L}_h W_\alpha$ so that the system is completely invariant under the addition of a Killing $1$-form to $W_\alpha$.

\medskip Then we have the following result~:

\begin{thm}\label{thm2}
Such a sequence $\left(u_\alpha,W_\alpha\right)$ is uniformly bounded in $C^{1,\eta}$ for all $\eta>0$ and, after passing to a subsequence, $\left(u_\alpha,W_\alpha\right)\to \left(u_0,W_0\right)$ in $C^{1,\eta}$ where $\left(u_0,W_0\right)$ is a solution of the limiting system.
\end{thm}

Compactness and stability results have a long time history. A striking result was the recent complete proof of the compactness of the Yamabe equation by Khuri-Marques-Schoen \cite{KhuMaSc} together with its limitation by Brendle \cite{Bre} and Brendle-Marques \cite{BreMa}. As a remark an equation can be compact but unstable, namely sensitive to changes of the parameters in the equation, and this is the case for the 
Yamabe equation on the projective space. General references on the stability of elliptic PDEs are by Druet \cite{DruetENSAIOS} and Hebey \cite{HebeyZLAM}. The specific case of the Einstein-Lichnerowicz equation is addressed in Druet-Hebey \cite{DruHeb}, Hebey-Veronelli \cite{HeVer} and Premoselli \cite{Premoselli2}. Changing the viewpoint, passing from the elliptic world to the dynamical setting, we mention the striking groundbreaking global stability of the Minkowski space-time obtained by Christodoulou-Klainerman \cite{ChristodoulouKlainerman} or the amount of work concerning the stability of Kerr space-times (see Dafermos-Rodnianski \cite{DafermosRodnianski} and the references therein). 

\medskip Theorem \ref{Th1} is a direct consequence of Theorem \ref{thm2}. The rest of the paper is devoted to the proof of Theorem \ref{thm2}. By standard elliptic theory, if the convergences in (\ref{eqconv}) take place in more regular spaces, one would get more regularity in the convergence in the conclusion of Theorem \ref{thm2}.

\section{Notations}\label{sectionnotations}

Given $P\in {\mathbb S}^3$, we let $\pi_P$ be the stereographic projection of pole $-P$. Then we have that 
$$\left(\pi_P^{-1}\right)^\star h= U^4 \xi\hskip.1cm,$$
where $\xi$ is the Euclidean metric and 
\begin{equation}\label{eqdefconffactor}
U(x)=\sqrt{\frac{2}{1+\vert x\vert^2}}\hskip.1cm.
\end{equation}
Given $u\in C^\infty\left({\mathbb S}^3\right)$ and $W$ a smooth $1$-form on $S^3$, we shall denote by $u\left[P\right]$ and $W\left[P\right]$ the following function and $1$-form on ${\mathbb R}^3$~:
\begin{equation}\label{eqfctstereo}
u\left[P\right](x) = u\circ \pi_P^{-1}(x)
\end{equation}
and 
\begin{equation}\label{eqvectstereo}
\left(\pi_P\right)^\star W\left[P\right]=W\hskip.1cm.
\end{equation}
Then we have that 
\begin{equation}\label{eqlapfctstereo}
\Delta_\xi\left(u[P]U\right)= U^5\left(\Delta_h u +\frac{3}{4}u\right)\circ \pi_P^{-1}\hskip.1cm,
\end{equation}
that 
\begin{equation}\label{eqLgstereo}
\left(\pi_P\right)^\star \left(U^4{\mathcal L}_\xi \left(U^{-4}W[P]\right) \right)=  {\mathcal L}_h W
\end{equation}
and that 
\begin{equation}\label{eqlapvectstereo}
\left(\pi_P\right)^\star\left(\overrightarrow{\Delta_\xi} \left(U^{-4}W[P]\right)_i - 6U^{-1}{\mathcal L}_\xi \left(U^{-4}W[P]\right)_{ij}\partial^j U\right)= \left(\overrightarrow{\Delta_h} W\right)_i\hskip.1cm.
\end{equation}

\section{Proof of Theorem \ref{thm2}}\label{sectionproof}

We let $\left(u_\alpha,W_\alpha\right)$ be a sequence of solutions on $\left(S^3,h\right)$ of $\left(\ref{equationEalpha}\right)$. And we assume that (\ref{eqconv}) and (\ref{eqnonzero}) hold. We first claim that $u_\alpha$ stays uniformly positive~:

\begin{claim}\label{claim1}
There exists $\eps_0>0$ such that $u_\alpha\ge \eps_0$ on ${\mathbb S}^3$ for all $\alpha$.
\end{claim}

\medskip {\bf Proof.} We let $G_\alpha$ be the Green function of $\Delta_g + h_\alpha$ on ${\mathbb S}^3$ which is uniformly positive thanks to (\ref{eqconv}) and the fact that $\Delta_h + h_0$ is coercive (see \cite{RobWeb}). Then we can use Green's representation formula to write that 
$$u_\alpha(x) \ge C \int_{{\mathbb S}^3} \left(f_\alpha u_\alpha^5+a_\alpha u_\alpha^{-7}\right)\, dv_h \hskip.1cm.$$
Since 
$$f_\alpha u_\alpha^5+a_\alpha u_\alpha^{-7}\ge f_\alpha u_\alpha^5+b_\alpha u_\alpha^{-7}\ge \frac{12b_\alpha}{5}\left(\frac{7b_\alpha}{5f_\alpha}\right)^{-\frac{7}{12}}\hskip.1cm,$$
the claim follows from the assumptions $f_0>0$ and $b_0\not\equiv 0$. \hfill $\diamondsuit$

\begin{claim}\label{claim2}
If there exists $C>0$ such that $u_\alpha\le C$ on ${\mathbb S}^3$ for all $\alpha$, then, for all $0<\eta<1$, there exists $D>0$ such that 
$$\left\Vert W_\alpha\right\Vert_{C^{1,\eta}} + \left\Vert u_\alpha\right\Vert_{C^{1,\eta}}\le D\hskip.1cm.$$
In particular, the conclusion of Theorem \ref{thm2} holds in this case.
\end{claim}

\medskip {\bf Proof.} This follows from standard elliptic theory as developed in section \ref{stddelltheory}, see proposition \ref{stdellipticclosed}. \hfill $\diamondsuit$ 

\medskip From now on, we assume that 
\begin{equation}\label{hypblowup}
\sup_{{\mathbb S}^3} u_\alpha\to +\infty \hbox{ as }\alpha\to +\infty\hskip.1cm.
\end{equation}

\begin{claim}\label{claim3}
There exist $N_\alpha\in {\mathbb N}^\star$ and $\left(x_{1,\alpha}, x_{2,\alpha},\dots, x_{N_\alpha,\alpha}\right)$ points in ${\mathbb S}^3$ such that~:

\smallskip i) $\nabla u_\alpha\left(x_{i,\alpha}\right)=0$ for $i=1,\dots,N_\alpha$.

\smallskip ii) $d_h\left(\xia,\xja\right)^{\frac{1}{2}}u_\alpha\left(\xia\right)\ge 1$ for all $i,j\in\left\{1,\dots,N_\alpha\right\}$, $i\neq j$.

\smallskip iii) There exists $C>0$ such that 
$$\left(\min_{i=1,\dots,N_\alpha}d_h\left(\xia,x\right)\right)^{3}\left(u_\alpha(x)^6+\left\vert {\mathcal L}_h W_\alpha\right\vert_h(x)\right)\le C$$
for all $\alpha$ and all $x\in {\mathbb S}^3$.
\end{claim}

\medskip {\bf Proof.} We start with lemma 1.1 of \cite{DruHeb} which gives us $N_\alpha\in {\mathbb N}^\star$ and points $\left(x_{1,\alpha}, x_{2,\alpha},\dots, x_{N_\alpha,\alpha}\right)$  in ${\mathbb S}^3$ such that~:

\smallskip i) $\nabla u_\alpha\left(x_{i,\alpha}\right)=0$ for $i=1,\dots,N_\alpha$.

\smallskip ii) $d_h\left(\xia,\xja\right)^{\frac{1}{2}}u_\alpha\left(\xia\right)\ge 1$ for all $i,j\in\left\{1,\dots,N_\alpha\right\}$, $i\neq j$.

\smallskip iv) For any $x\in {\mathbb S}^3$ with $\nabla u_\alpha(x)=0$,  
$$ \left(\min_{i=1,\dots,N_\alpha}d_h\left(\xia,x\right)\right)^{\frac{1}{2}}u_\alpha(x)\le 1\hskip.1cm.$$

\medskip\noindent The aim is to replace iv) by iii). Assume by contradiction that iii) is false and let $x_\alpha\in {\mathbb S}^3$ be such that 
\begin{equation}\label{eq-claim3-1}
\Phi_\alpha\left(x_\alpha\right)= \sup_{{\mathbb S}^3} \Phi_\alpha(x)\to +\infty \hbox{ as }\alpha\to +\infty\hskip.1cm,
\end{equation}
where
$$\Phi_\alpha(x)=\left(\min_{i=1,\dots,N_\alpha}d_h\left(\xia,x\right)\right)^{3}\left(u_\alpha(x)^6+\left\vert {\mathcal L}_h W_\alpha\right\vert_h(x)\right)\hskip.1cm.$$
We let $\mu_\alpha>0$ be such that  
\begin{equation}\label{eq-claim3-2}
u_\alpha\left(\xa\right)^6+\left\vert {\mathcal L}_h W_\alpha\right\vert_h\left(x_\alpha\right)=\mu_\alpha^{-3}\hskip.1cm.
\end{equation}
Since ${\mathbb S}^3$ is compact, it is then clear that 
\beq\label{eq-claim3-3}
\ma\to 0\hbox{ as }\alpha\to +\infty\hskip.1cm.
\eeq
We also have thanks to (\ref{eq-claim3-1}) that
\beq\label{eq-claim3-4}
\frac{d_h\left(x_\alpha,{\mathcal S}_\alpha\right)}{\ma}\to +\infty\hbox{ as }\alpha\to +\infty\hskip.1cm,
\eeq
where 
$${\mathcal S}_\alpha = \left\{x_{1,\alpha},\dots,x_{N_\alpha,\alpha}\right\}\hskip.1cm.$$
We set 
\begin{equation}\label{eq-claim3-5}
\tilde{u}_\alpha = \ma^{\frac{1}{2}}u_\alpha\left[x_\alpha\right]\left(\mu_\alpha x\right) U\left(\mu_\alpha x\right)
\end{equation}
and
\begin{equation}\label{eq-claim3-6}
\tilde{W}_\alpha = \mu_\alpha^2 U\left(\mu_\alpha x\right)^{-4} W_\alpha\left[x_\alpha\right]\left(\mu_\alpha x\right)\hskip.1cm.
\end{equation}
Then, equation (\ref{equationEalpha}) leads with (\ref{eqlapfctstereo}), (\ref{eqLgstereo}) and (\ref{eqlapvectstereo}) to 
\begin{equation}\label{eq-claim3-7}
\left\{\begin{array}{l}
{\ds \Delta_\xi \tilde{u}_\alpha + \ma^2 \tilde{h}_\alpha \tilde{u}_\alpha = \tilde{f}_\alpha \tilde{u}_\alpha^5 + \frac{\tilde{a}_\alpha}{\tilde{u}_\alpha^7}}\\
\,\\
{\ds \left(\overrightarrow{\Delta_\xi} \tilde{W}_\alpha\right)_i = -6\mu_\alpha^2 \frac{x^j}{1+\mu_\alpha^2 \vert x\vert^2} \left({\mathcal L}_\xi \tilde{W}_\alpha\right)_{ij} + \mu_\alpha \tilde{u}_\alpha^{6}\left(\tilde{X}_\alpha\right)_i +\mu_\alpha^4\left(\tilde{Y}_\alpha\right)_i}\\
\end{array}\right. 
\end{equation}
where 
\begin{equation}\label{eq-claim3-8}
\begin{array}{l}
{\ds \tilde{h}_\alpha (x)= \left(h_\alpha\left[x_\alpha\right]\left(\mu_\alpha x\right)-\frac{3}{4}\right)U\left(\mu_\alpha x\right)^4\hskip.1cm,}\\
\,\\{\ds \tilde{f}_\alpha (x)= f_\alpha\left[x_\alpha\right]\left(\mu_\alpha x\right)\hskip.1cm,}\\
\,\\{\ds \tilde{a}_\alpha(x)= \left\vert \mu_\alpha^3 \tilde{U}_\alpha+U\left(\mu_\alpha x\right)^6 {\mathcal L}_\xi \tilde{W}_\alpha\right\vert_\xi^2+\mu_\alpha^6 U\left(\mu_\alpha x\right)^{12} b_\alpha\left[x_\alpha\right]\left(\mu_\alpha x\right)\hskip.1cm,}\\
\,\\{\ds \tilde{U}_\alpha(x)=U\left(\mu_\alpha x\right)^2 \left(\pi_{x_\alpha}\right)_\star U_\alpha\left(\mu_\alpha x\right)\hskip.1cm,}\\
\,\\{\ds \tilde{X}_\alpha (x) = U\left(\mu_\alpha x\right)^{-6}\left(\pi_{x_\alpha}\right)_\star X_\alpha \left(\mu_\alpha x\right)\hskip.1cm,}\\
\,\\{\ds \tilde{Y}_\alpha(x)= \left(\pi_{x_\alpha}\right)_\star Y_\alpha \left(\mu_\alpha x\right)\hskip.1cm.}\\
\end{array}
\end{equation}
We know thanks to (\ref{eq-claim3-1}), (\ref{eq-claim3-2}), (\ref{eq-claim3-4}), (\ref{eq-claim3-5}) and (\ref{eq-claim3-6})  that 
\begin{equation}\label{eq-claim3-9}
\sup_{B_0\left(R\right)}\left(\frac{1}{8}\tilde{u}_\alpha^6+\left\vert {\mathcal L}_\xi \tilde{W}_\alpha\right\vert_\xi\right) \le 1+o(1)
\end{equation}
for all $R>0$ and that 
\begin{equation}\label{eq-claim3-10}
\frac{1}{8}\tilde{u}_\alpha(0)^6+\left\vert {\mathcal L}_\xi \tilde{W}_\alpha\right\vert_\xi (0)=1\hskip.1cm.
\end{equation}
Given $y\in {\mathbb R}^3$ and $R>0$, let us use the Green representation formula to write that 
$$\tilde{u}_\alpha(y)\ge \frac{1}{4\pi}\int_{B_y(2R)} \left(\frac{1}{\vert x- y\vert}-\frac{1}{2R}\right) \Delta_\xi \tilde{u}_\alpha\, dx$$
since $\tilde{u}_\alpha \ge 0$. Using equation (\ref{eq-claim3-7}) and the fact that $\tilde{f}_\alpha > 0$, we get that 
\begincal
\tilde{u}_\alpha(y)&\ge& \frac{1}{4\pi} \int_{B_y(2R)} \left(\frac{1}{\vert x- y\vert}-\frac{1}{2R}\right) \frac{\tilde{a}_\alpha(x)}{\tilde{u}_\alpha(x)^7}\, dx \\
&&-\frac{\mu_\alpha^2}{4\pi} \int_{B_y(R)} \left(\frac{1}{\vert x- y\vert}-\frac{1}{2R}\right) \tilde{h}_\alpha(x)\tilde{u}_\alpha(x)\, dx\hskip.1cm.
\fincal
We deduce that 
$$\int_{B_y(R)} \left\vert x-y\right\vert^{-1} \tilde{a}_\alpha(x)\le 8\pi \left(\sup_{B_y\left(2R\right)}\tilde{u}_\alpha\right)^8\left(1 +2 R^2\mu_\alpha^2 \sup_{B_y\left(2R\right)}\left\vert \tilde{h}_\alpha\right\vert\right)\hskip.1cm.$$
Using (\ref{eq-claim3-9}), we get that 
$$\limsup_{\alpha\to +\infty}\int_{B_y(R)} \left\vert x-y\right\vert^{-1} \tilde{a}_\alpha(x)\, dx \le C_1$$
for any $R>0$ and any $y\in {\mathbb R}^3$ where $C_1>0$ is some constant independent of $R$ and $y$. Thanks to (\ref{eqconv}) and (\ref{eq-claim3-8}), this leads to the existence of some $C_2>0$ independent of $R>0$ and $y\in {\mathbb R}^3$ such that 
$$\limsup_{\alpha\to +\infty}\int_{B_y(R)\setminus B_y\left(\frac{R}{2}\right)} \left\vert x-y\right\vert^{-1}\left\vert {\mathcal L}_\xi \tilde{W}_\alpha\right\vert_\xi^2(x)\, dx \le C_2$$
for all $R>0$ and all $y\in {\mathbb R}^3$. We deduce easily that~: for any $y\in {\mathbb R}^3$ and for any $R>0$, there exists $\frac{R}{2}\le r_\alpha\le R$ such that 
\begin{equation}\label{eq-claim3-11}
\int_{\partial B_y\left(r_\alpha\right)} \left\vert {\mathcal L}_\xi \tilde{W}_\alpha\right\vert_\xi^2(x)\, dx \le \frac{2C_2 r_\alpha}{R}\hskip.1cm.
\end{equation}
We use now the Green representation formula, see Proposition \ref{vecteurGreen}, to write that 
\begincal
\left\vert {\mathcal L}_\xi \tilde{W}_\alpha\right\vert (y)&\le & C_3 \int_{B_y\left(r_\alpha\right)} \left\vert x-y\right\vert^{-2} \left\vert \overrightarrow{\Delta_\xi} \tilde{W}_\alpha(x)\right\vert\, dx + C_3 r_\alpha^{-2}\int_{\partial B_y\left(r_\alpha\right)}\left\vert {\mathcal L}_\xi \tilde{W}_\alpha\right\vert \, d\sigma \\
&\le &C_3 \int_{B_y\left(r_\alpha\right)} \left\vert x-y\right\vert^{-2} \left\vert \overrightarrow{\Delta_\xi} \tilde{W}_\alpha(x)\right\vert\, dx + \frac{4C_3\sqrt{\pi C_2}}{R}\\
\fincal
thanks to (\ref{eq-claim3-11}) and to the fact that $2r_\alpha\ge R$. Using (\ref{eqconv}), (\ref{eq-claim3-7}), (\ref{eq-claim3-8}) and (\ref{eq-claim3-9}), we get that 
$$ \int_{B_y\left(r_\alpha\right)} \left\vert x-y\right\vert^{-2} \left\vert \overrightarrow{\Delta_\xi} \tilde{W}_\alpha(x)\right\vert\, dx \to 0$$
as $\alpha\to +\infty$. Thus we obtain that 
$$\limsup_{\alpha\to +\infty} \left\vert {\mathcal L}_\xi \tilde{W}_\alpha\right\vert (y)\le \frac{4C_3\sqrt{\pi C_2}}{R}\hskip.1cm.$$
Since this holds for all $R>0$, we have proved that 
\begin{equation}\label{eq-claim3-12}
 {\mathcal L}_\xi \tilde{W}_\alpha\to 0\hbox{ in }L^\infty_{loc}\left({\mathbb R}^3\right)\hbox{ as }\alpha\to +\infty\hskip.1cm.
 \end{equation}
Then we have thanks to (\ref{eqconv}) and (\ref{eq-claim3-8}) that 
$$\mu_\alpha^2 \tilde{h}_\alpha\to 0,\, \tilde{f}_\alpha\to f_0\left(x_0\right)\hbox{ and }\tilde{a}_\alpha\to 0\hbox{ in }L^\infty_{loc}\left({\mathbb R}^3\right)\hbox{ as }\alpha\to +\infty$$
where $x_\alpha\to x_0$ as $\alpha\to +\infty$ (up to a subsequence). By (\ref{eq-claim3-10}) and (\ref{eq-claim3-12}), there holds ${\ds \tilde{u}_\alpha(0)=\sqrt{2}+o(1)}$. Thanks to (\ref{eq-claim3-9}) and (\ref{eq-claim3-12}), the Harnack inequality of Proposition \ref{prop-HI} shows that $\left(\tilde{u}_\alpha\right)$ is uniformly positive in any compact set of ${\mathbb R}^3$ so that we can pass to the limit in equation (\ref{eq-claim3-7}) to get that, after passing to a subsequence, 
\begin{equation}\label{eq-claim3-13}
\tilde{u}_\alpha\to \tilde{u}\hbox{ in } C^{1,\eta}_{loc}\left({\mathbb R}^3\right)\hbox{ as }\alpha\to +\infty\hskip.1cm,
\end{equation}
where 
$$\Delta_\xi \tilde{u} = f_0\left(x_0\right) \tilde{u}^5\hbox{ in }{\mathbb R}^3$$
and 
$$0< \tilde{u}\le \tilde{u}(0)=\sqrt{2}$$
thanks to (\ref{eq-claim3-9}), (\ref{eq-claim3-10}) and (\ref{eq-claim3-12}). Then, by the classification result of Caffarelli-Gidas-Spruck \cite{CaGiSp}, we have that 
$$\tilde{u}(x) =  U\left(\sqrt{\frac{4f_0\left(x_0\right)}{3}}x\right)$$
so that it possesses a strict maximum in $0$. Thus $\tilde{u}_\alpha$ necessarily possesses a critical point in the neighbourhoud of $0$ for $\alpha$ large, which means that $u_\alpha$ possesses a critical point $y_\alpha\in {\mathbb S}^3$ with $d_h\left(y_\alpha,x_\alpha\right)=o\left(\mu_\alpha\right)$. This critical point $y_\alpha$ clearly violates point (iv). This proves that (iii) must hold and this end the proof of the claim. \hfill $\diamondsuit$

\medskip For the following claims (\ref{claim4} and \ref{claim5}), we consider $x_\alpha\in {\mathbb S}^3$ such that $\nabla u_\alpha\left(x_\alpha\right)=0$ and $\rho_\alpha>0$ such that there exists $C>0$ such that 
\beq\label{eq-ba-1}
d_h\left(x_\alpha,x\right)^{3} \left(u_\alpha(x)^6+\left\vert {\mathcal L}_h W_\alpha\right\vert_h (x)\right)\le C\hbox{ for all }x\in B_{\xa}\left(8\rho_\alpha\right)
\eeq
and such that 
\beq\label{eq-ba-2}
\rho_\alpha^{\frac{1}{2}} u_\alpha\left(x_\alpha\right)\ge \frac{1}{C}\hskip.1cm.
\eeq
This will be the case for any sequence $\left(x_{i_\alpha,\alpha}\right)$ with $\rho_\alpha = \frac{1}{16}d_h\left(x_{i_\alpha,\alpha}, {\mathcal S}_\alpha\setminus \left\{x_{i_\alpha,\alpha}\right\}\right)$ as in claim \ref{claim3}.

\begin{claim}\label{claim4}
If 
$$\rho_\alpha^3\sup_{B_{x_\alpha}\left(8\rho_\alpha\right)} \left(u_\alpha^6 + \left\vert {\mathcal L}_h W_\alpha\right\vert_h\right)\le C_1$$
for some $C_1>0$, then there exists $C_2>0$ such that 
$$\rho_\alpha^{\frac{1}{2}}u_\alpha\ge C_2 \hbox{ in }B_{\xa}\left(4\rho_\alpha\right)\hskip.1cm.$$
\end{claim}

\medskip {\it Proof }- If $\rho_\alpha\not\to 0$ as $\alpha\to +\infty$, it is a simple consequence of Claim \ref{claim1}. If $\rho_\alpha\to 0$ as $\alpha\to +\infty$, we want to apply Harnack's inequality to $\rho_\alpha^{\frac{1}{2}}u_\alpha\left[x_\alpha\right]\left(\rho_\alpha\, \cdot\,\right)$. It is uniformly bounded in the $\rho_\alpha^{-1}$-dilated image by the stereographic projection of the ball $B_{x_\alpha}\left(8\rho_\alpha\right)$ and ${\ds \rho_\alpha^{\frac{1}{2}}u_\alpha\left[x_\alpha\right](0)\ge \frac{\sqrt{2}}{C}}$ thanks to the assumption of the claim. In this same ball, $\rho_\alpha^3 \left\vert {\mathcal L}_\xi W_\alpha\left[x_\alpha\right]\right\vert$ is also uniformly bounded thanks to the assumption of the claim. Thus we can use Harnack's inequality, see proposition \ref{prop-HI}, together with (\ref{eqconv}) to conclude. \hfill $\diamondsuit$

\begin{claim}\label{claim5}
If 
$$\rho_\alpha^3\sup_{B_{x_\alpha}\left(8\rho_\alpha\right)} \left(u_\alpha^6 + \left\vert {\mathcal L}_h W_\alpha\right\vert_h\right)\to +\infty \hbox{ as }\alpha\to +\infty\hskip.1cm,$$
then $\rho_\alpha\to 0$ as $\alpha\to +\infty$ and 
$$\sup_{B_0\left(\rho_\alpha\right)} \left\vert \frac{u_\alpha\left[x_\alpha\right]U}{B_\alpha}- 1 \right\vert \to 0\hbox { as }\alpha\to +\infty\hskip.1cm,$$
where 
$$B_\alpha(x) = \sqrt{2}\mu_\alpha^{\frac{1}{2}}\left(\mua^2 + \frac{4f_\alpha\left(x_\alpha\right)}{3} \left\vert x\right\vert^2\right)^{-\frac{1}{2}}$$
with 
$$u_\alpha\left(x_\alpha\right)=\ma^{-\frac{1}{2}}\to +\infty \hbox{ as }\alpha\to +\infty$$
and ${\ds \frac{\rho_\alpha}{\mu_\alpha}\to +\infty}$ as $\alpha\to +\infty$.
\end{claim}

\medskip{\it Proof} - We postpone the proof of this claim to Section \ref{sectionlba}. This is the core of the analysis of this paper. 

\medskip We are now in position to conclude the proof of Theorem \ref{thm2}. We let 
\beq\label{eq-concl-1}
16 d_\alpha = \min_{i\neq j} d_h\left(\xia,\xja\right)
\eeq
and we assume that 
\begin{equation}\label{eq-concl-2}
d_\alpha\to 0\hbox{ as }\alpha\to +\infty\hskip.1cm,
\end{equation}
where the $\xia$'s, $i=1,\dots, N_\alpha$, are those of claim \ref{claim3} and we assume, up to reordering, that 
$$d_h\left(x_{1,\alpha}, x_{2,\alpha}\right)=16 d_\alpha\hskip.1cm.$$
Note that, by Claim \ref{claim5} above and thanks to Claim \ref{claim1}, we know that $N_\alpha\ge 2$ for $\alpha$ large. We let 
$$\tua(x) = d_\alpha^{\frac{1}{2}}u_\alpha\left[x_{1,\alpha}\right]\left(d_\alpha x\right)U\left(d_\alpha x\right)\hskip.1cm.$$
Equation (\ref{equationEalpha}) becomes 
\begin{equation}\label{eq-concl-3}
\Delta_\xi \tua + d_\alpha^2 \tha \tua =\tfa \tua^{5} +\taa\tua^{-7}\hskip.1cm,
\end{equation}
where
$$\taa = \bigl(64+o(1)\bigr) \left\vert {\mathcal L}_\xi \tWa\right\vert^2+o(1) $$
in any compact set with 
\begincal
\tWa(x)&=& d_\alpha^2 U\left(d_\alpha x\right)^{-4} W_\alpha\left[x_{1,\alpha}\right]\left(d_\alpha x\right)\hskip.1cm,\\
\tfa(x)&=& f_\alpha\left[x_\alpha\right]\left(d_\alpha x\right)\hskip.1cm,\\
\tha(x)&=& h_\alpha\left[x_\alpha\right]\left(d_\alpha x\right)\hskip.1cm.
\fincal
We shall also let 
$$\tilde{x}_{i,\alpha}=d_\alpha^{-1}\pi_{x_{1,\alpha}}\left(x_{i,\alpha}\right)$$
for $i=1,\dots, N_\alpha$. Let us fix $R>0$ large. We reorder the concentration points such that 
$$\left\vert \tilde{x}_{i,\alpha}\right\vert \le R \hbox{ if and only if }1\le i\le N_R$$
for some $N_R\ge 2$ fixed (up to a subsequence and for $\alpha$ large). For any $i\in \left\{1,\dots,N_R\right\}$, we have the following alternative, thanks to Claims \ref{claim4} and \ref{claim5}~: 
\begin{equation}\label{eq-alternative}
\begin{array}{l}
{\ds \hbox{ either there exists }C_i>0\hbox{ such that } C_i^{-1}\le \tua \le C_i\hbox{ in }B_{\tilde{x}_{i,\alpha}}\left(\frac{1}{2}\right)}\\
{\ds \quad \hbox{ or } \sup_{B_{\tilde{x}_{i,\alpha}}\left(\frac{1}{2}\right)} \left\vert \frac{\tua}{\tilde{B}_{i,\alpha}}-1\right\vert \to 0\hbox{ as }\alpha\to +\infty\hskip.1cm,}
\end{array}
\end{equation}
where 
$$\tilde{B}_{i,\alpha}(x)= \sqrt{2}\mu_{i,\alpha}^{\frac{1}{2}}\left(\mu_{i,\alpha}^2 + \frac{4f_\alpha\left(x_{i,\alpha}\right)}{3}\vert x-\tilde{x}_{i,\alpha}\vert^2\right)^{-\frac{1}{2}}$$
with $\sqrt{2}\mu_{i,\alpha}^{-\frac{1}{2}}=\tua\left(\tilde{x}_{i,\alpha}\right)$. These two alternatives are exclusive one to each other and the second one appears if and only if $\tua\left(\tilde{x}_{i,\alpha}\right)\to +\infty$ as $\alpha\to +\infty$. 

Using the Green's representation formula and (\ref{eq-concl-3}), we can write that 
\begincal
\tilde{u}_\alpha\left(x\right)& \ge & \frac{1}{4\pi} \int_{B_{x}(10R)}\left(\frac{1}{\left\vert x-y\right\vert}-\frac{1}{10R}\right) \tfa(y) \tua(y)^5\, dy\\
&& -\frac{d_\alpha^2}{4\pi}\int_{B_{x}(10R)}\left(\frac{1}{\left\vert x-y\right\vert}-\frac{1}{10R}\right) \tha(y) \tua(y)\, dy
\fincal
since $\tfa>0$ and $\tua >0$. By the dominated convergence theorem, since 
$$\tilde{u}_\alpha \le D_1 d_\xi\left(x,\left\{\tilde{x}_{i,\alpha}\right\}\right)^{-\frac{1}{2}}$$
for some $D_1>0$, we have that there exists $D_2>0$ which does not depend on $\alpha$ and $x$ such that  
$$\frac{d_\alpha^2}{4\pi}\int_{B_x\left(10 R\right)}\left(\frac{1}{\left\vert x-y\right\vert}-\frac{1}{10 R}\right) \tha(y) \tua(y)\, dy\le D_2 d_\alpha^2$$
for all $x$ in $B_0(R)$. Thus the above becomes 
\begin{equation}\label{eq-concl-3bis}
\tilde{u}_\alpha\left(x\right) \ge \frac{1}{4\pi} \int_{B_{x}(10 R)}\left(\frac{1}{\left\vert x-y\right\vert}-\frac{1}{10 R} \right)\tfa(y) \tua(y)^5\, dy - D_2 d_\alpha^2
\end{equation}
for all $x\in B_0(R)$.

If the first alternative in (\ref{eq-alternative}) holds for some $i\in \left\{1,\dots,N_R\right\}$, then we can write that 
\begincal
&&\frac{1}{4\pi} \int_{B_{x}(10 R)}\left(\frac{1}{\left\vert x-y\right\vert}-\frac{1}{10 R} \right)\tfa(y) \tua(y)^5\, dy\\
&&\quad \ge \frac{1}{4\pi}C_i^{-5} \int_{B_{\tilde{x}_{i,\alpha}}\left(\frac{1}{2}\right)}\left(\frac{1}{\left\vert x-y\right\vert}-\frac{1}{10 R} \right)\tfa(y) \, dy\\
&&\quad\ge  D_3
\fincal
for some $D_3>0$ independent of $\alpha$ and $x$ for all $x\in B_0(R)$ since $\tfa$ is uniformly positive. This implies thanks to (\ref{eq-concl-2}) and (\ref{eq-concl-3bis}) that 
$$\tilde{u}_\alpha\ge \frac{1}{2} D_3$$
for $\alpha$ large for all $x\in B_0(R)$. This proves that the second alternative in (\ref{eq-alternative}) can not happen for any $j\in \left\{1,\dots,N_R\right\}$. 

So far we have proved that either the first alternative in (\ref{eq-alternative}) holds for all $i\in \left\{1,\dots,N_R\right\}$, in which case $\tua$ is uniformly bounded in $B_0\left(\frac{R}{2}\right)$, or that the second alternative in (\ref{eq-alternative}) holds for all $i\in \left\{1,\dots,N_R\right\}$. 

We claim now that we are in the first situation, namely that the second alternative in (\ref{eq-alternative}) can not hold for all $i\in \left\{1,\dots,N_R\right\}$. Indeed, assume it holds for some $i\in \left\{1,\dots,N_R\right\}$, then it holds for all of them. Then, we can use the estimate of the second alternative in (\ref{eq-alternative}) to write that 
$$\frac{1}{4\pi} \int_{B_{\tilde{x}_{i,\alpha}}\left(\frac{1}{2}\right)}\left(\frac{1}{\left\vert x-y\right\vert}-\frac{1}{10 R} \right)\tfa(y) \tua(y)^5\, dy\ge \left(1+o(1)\right) \tilde{B}_{i,\alpha}(x) - \frac{D_4}{R} \mu_{i,\alpha}^{\frac{1}{2}}$$
for $i=1,2$ where $D_4>0$ is some constant independent of $R$, $\alpha$ and $x$. This proves thanks to (\ref{eq-concl-3bis}) that 
$$\tilde{u}_\alpha(x) \ge \left(1+o(1)\right) \left( \tilde{B}_{1,\alpha}(x) +\tilde{B}_{2,\alpha}(x)\right)-\frac{D_4}{R} \left(\sqrt{\mu_{1,\alpha}}+\sqrt{\mu_{2,\alpha}}\right) - D_2 d_\alpha^2\hskip.1cm.$$
Applying this to some $x$ such that $\vert x\vert \le \frac{1}{4}$, we get thanks to the second alternative of (\ref{eq-alternative}) for $i=1$ that 
$$1+ o(1) \ge \left(1+ o(1)\right)\left(1 +\frac{\tilde{B}_{2,\alpha}(x)}{\tilde{B}_{1,\alpha}(x)}\right) - \frac{D_4}{R} \frac{\left(\sqrt{\mu_{1,\alpha}}+\sqrt{\mu_{2,\alpha}}\right)}{\tilde{B}_{1,\alpha}(x)}-D_2 \frac{d_\alpha^2}{\tilde{B}_{1,\alpha}(x)}\hskip.1cm.$$
Thus we get that 
$$\frac{\tilde{B}_{2,\alpha}(x)}{\tilde{B}_{1,\alpha}(x)} \le o(1) + \frac{D_4}{R} \frac{\left(\sqrt{\mu_{1,\alpha}}+\sqrt{\mu_{2,\alpha}}\right)}{\tilde{B}_{1,\alpha}(x)}-D_2 \frac{d_\alpha^2}{\tilde{B}_{1,\alpha}(x)}\hskip.1cm.$$
If $x\neq 0$, we have that 
$$\frac{\tilde{B}_{2,\alpha}(x)}{\tilde{B}_{1,\alpha}(x)} = \left(1+o(1)\right)\frac{\vert x\vert \sqrt{\mu_{2,\alpha}}}{\vert x- \tilde{x}_2\vert \sqrt{\mu_{1,\alpha}}}\hskip.1cm,$$
where $\tilde{x}_2=\lim_{\alpha\to +\infty} \tilde{x}_{2,\alpha}$ is such that $\left\vert \tilde{x}_2\right\vert=16$. We also have that 
$$\frac{\left(\sqrt{\mu_{1,\alpha}}+\sqrt{\mu_{2,\alpha}}\right)}{\tilde{B}_{1,\alpha}(x)} \le D_5 \vert x\vert\left(1+\frac{\sqrt{\mu_{2,\alpha}}}{\sqrt{\mu_{1,\alpha}}}\right)$$
for some $D_5>0$ independent of $x$ and $\alpha$. At last, we know that 
$$\frac{d_\alpha^2}{\tilde{B}_{1,\alpha}(x)}\to 0\hbox{ as }\alpha\to +\infty$$
thanks to claim \ref{claim1} and the second alternative in (\ref{eq-alternative}) which tells us that $d_\alpha =O\left(\mu_{1,\alpha}\right)$. Thus we arrive to 
$$\frac{\vert x\vert \sqrt{\mu_{2,\alpha}}}{\vert x- \tilde{x}_2\vert \sqrt{\mu_{1,\alpha}}} \le o(1) + \frac{D_4 D_5}{R} \vert x\vert\left(1+\frac{\sqrt{\mu_{2,\alpha}}}{\sqrt{\mu_{1,\alpha}}}\right) \left(1+o(1)\right)\hskip.1cm.$$
Up to choose $R$ large enough, this leads, letting $x\to 0$, to 
$$\limsup_{\alpha\to +\infty}\frac{\sqrt{\mu_{2,\alpha}}}{\sqrt{\mu_{1,\alpha}}}\le \frac{16 D_4 D_5}{R-16 D_4 D_5}\hskip.1cm.$$
Using the same arguments exchanging the role of $\tilde{x}_{1,\alpha}$ and $\tilde{x}_{2,\alpha}$, we obtain also that
$$\limsup_{\alpha\to +\infty}\frac{\sqrt{\mu_{1,\alpha}}}{\sqrt{\mu_{2,\alpha}}}\le \frac{16 D_4 D_5}{R-16 D_4 D_5}\hskip.1cm.$$
This clearly leads to a contradiction as soon as $R$ is chosen large enough so that the right-hand side is less than $1$. 

Thus we have proved that, up to choose $R$ large enough, only the first alternative in (\ref{eq-alternative}) can happen for all $i\in \left\{1,\dots,N_R\right\}$ and Claim \ref{claim3} and Proposition \ref{prop-HI} show that 
$$C_K^{-1}\le  \tua \le C_K$$
for all compact set $K$ of ${\mathbb R}^3$. 

Moreover, we know that ${\mathcal L}_\xi \tWa$ is uniformly bounded in any compact set of ${\mathbb R}^3$. Then we can prove as during claim \ref{claim3} that, after passing to a subsequence,
$${\mathcal L}_\xi \tWa \to 0\hbox{ in }C^0_{loc}\left({\mathbb R}^3\right)\hbox{ as }\alpha\to +\infty$$
and that 
$$\tua \to \tilde{u}\hbox{ in }C^1_{loc}\left({\mathbb R}^3\right)\hbox{ as }\alpha\to +\infty$$
with
$$\Delta_\xi \tilde{u} = f_0\left(x_1\right)\tilde{u}^5\hbox{ in }{\mathbb R}^3\hskip.1cm.$$
Here, ${\ds x_1 = \lim_{\alpha\to +\infty} x_{1,\alpha}}$. Thus $\tilde{u}$ has only one critical point in ${\mathbb R}^3$, thanks to the classification result of \cite{CaGiSp}, but it should have at least two since $0$ and $\tilde{x}_{2,\alpha}$ were critical points of $\tua$. This is absurd and we have contradicted the fact that $d_\alpha\to 0$ as $\alpha\to +\infty$. 

\medskip Thus $d_\alpha \ge \delta_0$ and there is only a finite number of concentration points and for each of them, it is clear that the sequence $\left(u_\alpha\right)$ remains bounded in $B_0\left(\delta_0\right)$. Otherwise, we would be in the situation of Claim \ref{claim5} and $d_\alpha$ would have to go to $0$. It is then clear thanks to this and (iii) of Claim \ref{claim3} that the sequence  $\left(u_\alpha\right)$ is uniformly bounded on ${\mathbb S}^3$. And this ends the proof of the theorem thanks to Claim \ref{claim2}.

\section{Local blow up analysis - Proof of claim \ref{claim4}}\label{sectionlba}

In this section, we perform the local blow up analysis needed to prove claim \ref{claim5}. We assume that we have a sequence $\left(u_\alpha,W_\alpha\right)$ of solutions of (\ref{equationEalpha}) with the convergences (\ref{eqconv}) and (\ref{eqnonzero}). We also assume that we have sequences $\left(x_\alpha\right)$ of critical points of $u_\alpha$ and $\left(\rho_\alpha\right)$ of positive real numbers with $0<\rho_\alpha\le \frac{\pi}{16}$ such that
\begin{equation}\label{eq-lba-1}
d_h\left(x_\alpha,x\right)^{3} \left(u_\alpha(x)^6+ \left\vert {\mathcal L}_h W_\alpha(x)\right\vert_h\right) \le C_1\hbox{ for all }x\in B_{\xa}\left(8\rho_\alpha\right)
\end{equation}
for some $C_1>0$ independent of $\alpha$. We assume moreover that 
\beq\label{eq-lba-2}
\rho_\alpha^{3}\sup_{B_{\xa}\left(8\rho_\alpha\right)}\left(u_\alpha^6 + \left\vert {\mathcal L}_h W_\alpha\right\vert_h\right)\to +\infty\hbox{ as }\alpha\to +\infty\hskip.1cm.
\eeq
In this section, the constants $C_i$'s will always denote constants independent of $\alpha$. The constants $D_i$' s also but they will have nothing to do one with the other when changing of claims. 

\begin{claim}\label{claim-lba1}
We have that, up to a subsequence,  
$$\mu_\alpha^{\frac{1}{2}} u_\alpha\left
[x_\alpha\right]\left(\mu_\alpha x\right)\to  \left(1+\frac{4f\left(x_0\right)}{3}\vert x\vert^2\right)^{-\frac{1}{2}}\hbox{ in }C^{1,\eta}_{loc}\left({\mathbb R}^3\right)\hbox{ as }\alpha\to +\infty$$
and that 
$$\mu_\alpha^3 {\mathcal L}_\xi W_\alpha\left[x_\alpha\right]\left(\mu_\alpha x\right)\to 0\hbox{ in } C^0_{loc}\left({\mathbb R}^3\right)\hbox{ as }\alpha\to +\infty\hskip.1cm,$$
where 
$$u_\alpha\left(x_\alpha\right)=\mu_\alpha^{-\frac{1}{2}} \to +\infty\hbox{ as }\alpha\to +\infty$$
and ${\ds \lim_{\alpha\to +\infty} x_\alpha=x_0}$.
\end{claim}

\medskip {\bf Proof.} It is really similar to that of claim \ref{claim3}. We let $y_\alpha\in {\mathbb S}^3$ be such that
\begin{equation}\label{eq-claim-lba1-1}
u_\alpha\left(y_\alpha\right)^6+\left\vert {\mathcal L}_h W_\alpha\right\vert_h\left(y_\alpha\right)= \sup_{B_{x_\alpha}\left(8\rho_\alpha\right)}\left(u_\alpha^6+\left\vert {\mathcal L}_h W_\alpha\right\vert_h\right)\hskip.1cm.
\end{equation}
And we set 
\begin{equation}\label{eq-claim-lba1-2}
u_\alpha\left(y_\alpha\right)^6+\left\vert {\mathcal L}_h W_\alpha\right\vert_h\left(y_\alpha\right)=\tmua^{-3}\hskip.1cm.
\end{equation}
By (\ref{eq-lba-2}), we know that 
\begin{equation}\label{eq-claim-lba1-3}
\frac{\rho_\alpha}{\tmua}\to +\infty \hbox{ as }\alpha\to +\infty\hskip.1cm.
\end{equation}
We also have thanks to (\ref{eq-lba-1}) that 
$$d_h\left(x_\alpha,y_\alpha\right)^3 u_\alpha\left(y_\alpha\right)^6+d_h\left(x_\alpha,y_\alpha\right)^3 \left\vert {\mathcal L}_h W_\alpha\left(y_\alpha\right)\right\vert_h \le C_1$$
so that 
\begin{equation}\label{eq-claim-lba1-4}
\frac{d_h\left(x_\alpha,y_\alpha\right)}{\tmua}\le C_1^{\frac{1}{3}}\hskip.1cm.
\end{equation}
We set 
\begin{equation}\label{eq-claim-lba1-5}
\tilde{u}_\alpha = \tmua^{\frac{1}{2}}u_\alpha\left[x_\alpha\right]\left(\tmua x\right) U\left(\tmua x\right)
\end{equation}
and 
\begin{equation}\label{eq-claim-lba1-6}
\tilde{W}_\alpha = \tmua^2 U\left(\tmua x\right)^{-4} W_\alpha\left[x_\alpha\right]\left(\tmua x\right)
\end{equation}
for $x\in {\mathbb R}^3$. We set 
\begin{equation}\label{eq-claim-lba1-7}
\tya = \tmua^{-1}\pi_{x_\alpha}\left(y_\alpha\right)
\end{equation}
and we know thanks to (\ref{eq-claim-lba1-4}) that 
\begin{equation}\label{eq-claim-lba1-8}
\left\vert \tya\right\vert = O\left(1\right)
\end{equation}
so that, after passing to a subsequence,
\begin{equation}\label{eq-claim-lba1-9}
\tya\to \tilde{y}_0\hbox{ as }\alpha\to +\infty\hskip.1cm.
\end{equation}
Thanks to (\ref{eqLgstereo}), (\ref{eq-claim-lba1-1}) and (\ref{eq-claim-lba1-2}), we know that 
\begin{equation}\label{eq-claim-lba1-10}
U\left(\tmua x\right)^{-6}\tua\left(x\right)^6 + \left\vert {\mathcal L}_\xi \tilde{W}_\alpha\right\vert_\xi(x)\le 1 = U\left(\tmua \tya\right)^{-6}\tua\left(\tya\right)^6 + \left\vert {\mathcal L}_\xi \tilde{W}_\alpha\right\vert_\xi\left(\tya\right)
\end{equation}
for all $x\in {\mathbb R}^3$ such that ${\ds d_h\left(x_\alpha,\pi_{x_\alpha}^{-1}\left(\mu_\alpha x\right)\right)\le 8\rho_\alpha}$. We also have since $x_\alpha$ is a critical point of $u_\alpha$ that 
\begin{equation}\label{eq-claim-lba1-11}
\nabla \tua\left(0\right)=0\hskip.1cm.
\end{equation}
Then, equation (\ref{equationEalpha}) leads with (\ref{eqlapfctstereo}), (\ref{eqLgstereo}) and (\ref{eqlapvectstereo}) to 
\begin{equation}\label{eq-claim-lba1-12}
\left\{\begin{array}{l}
{\ds \Delta_\xi \tua + \tmua^2 \tilde{h}_\alpha \tilde{u}_\alpha = \tilde{f}_\alpha \tilde{u}_\alpha^5 + \frac{\tilde{a}_\alpha}{\tilde{u}_\alpha^7}}\\
\,\\
{\ds \left(\overrightarrow{\Delta_\xi} \tilde{W}_\alpha\right)_i = -6\tmua^2 \frac{x^j}{1+\tmua^2 \vert x\vert^2} \left({\mathcal L}_\xi \tilde{W}_\alpha\right)_{ij} + \tmua  \tilde{u}_\alpha^{6}\left(\tilde{X}_\alpha\right)_i +\tmua^4\left(\tilde{Y}_\alpha\right)_i}\\
\end{array}\right. 
\end{equation}
where 
\begin{equation}\label{eq-claim-lba1-13}
\begin{array}{l}
{\ds \tilde{h}_\alpha (x)= \left(h_\alpha\left[x_\alpha\right]\left(\tmua x\right)-\frac{3}{4}\right)U\left(\tmua x\right)^4\hskip.1cm,}\\
\,\\{\ds \tilde{f}_\alpha (x)= f_\alpha\left[x_\alpha\right]\left(\tmua x\right)\hskip.1cm,}\\
\,\\{\ds \tilde{a}_\alpha(x)= \left\vert \tmua^3 \tilde{U}_\alpha+U\left(\tmua x\right)^6 {\mathcal L}_\xi \tilde{W}_\alpha\right\vert_\xi^2+\tmua^6 U\left(\tmua x\right)^{12} b_\alpha\left[x_\alpha\right]\left(\tmua x\right)\hskip.1cm,}\\
\,\\{\ds \tilde{U}_\alpha(x)=U\left(\tmua x\right)^2 \left(\pi_{x_\alpha}\right)_\star U_\alpha\left(\tmua x\right)\hskip.1cm,}\\
\,\\{\ds \tilde{X}_\alpha (x) = U\left(\tmua x\right)^{-6}\left(\pi_{x_\alpha}\right)_\star X_\alpha \left(\tmua x\right)\hskip.1cm,}\\
\,\\{\ds \tilde{Y}_\alpha(x)= \left(\pi_{x_\alpha}\right)_\star Y_\alpha \left(\tmua x\right)\hskip.1cm.}\\
\end{array}
\end{equation}
Given $y\in {\mathbb R}^3$ and $R>0$, let us use the Green representation formula to write that 
$$\tilde{u}_\alpha(y)\ge \frac{1}{4\pi}\int_{B_y(2R)} \left(\frac{1}{\vert x- y\vert}-\frac{1}{2R}\right) \Delta_\xi \tilde{u}_\alpha\, dx$$
since $\tua \ge 0$. Using equation (\ref{eq-claim-lba1-12}) and the fact that $\tilde{f}_\alpha \ge 0$, we get that 
\begincal
\tilde{u}_\alpha(y)&\ge &\frac{1}{8\pi} \int_{B_y(R)}\frac{1}{\vert x- y\vert} \frac{\tilde{a}_\alpha(x)}{\tilde{u}_\alpha(x)^7}\, dx\\
&&\quad -\frac{\tmua^2}{4\pi} \int_{B_y(R)}\left(\frac{1}{\vert x- y\vert}-\frac{1}{2R}\right) \tilde{h}_\alpha(x)\tilde{u}_\alpha(x)\, dx\hskip.1cm.
\fincal
We deduce that 
$$\int_{B_y(R)} \left\vert x-y\right\vert^{-1} \tilde{a}_\alpha(x)\le 8\pi \left(\sup_{B_y\left(R\right)}\tilde{u}_\alpha\right)^8\left(1 + 2R^2\tmua^2 \sup_{B_y\left(R\right)}\left\vert \tilde{h}_\alpha\right\vert\right)\hskip.1cm.$$
Using (\ref{eq-claim-lba1-10}), we get that 
$$\limsup_{\alpha\to +\infty}\int_{B_y(R)} \left\vert x-y\right\vert^{-1} \tilde{a}_\alpha(x)\, dx \le D_1$$
for any $R>0$ and any $y\in {\mathbb R}^3$ where $D_1>0$ is some constant independent of $R$ and $y$. Thanks to (\ref{eqconv}) and (\ref{eq-claim-lba1-13}), this leads to the existence of some $D_2>0$ independent of $R>0$ and $y$ such that 
$$\limsup_{\alpha\to +\infty}\int_{B_y(R)\setminus B_y\left(\frac{R}{2}\right)} \left\vert x-y\right\vert^{-1}\left\vert {\mathcal L}_\xi \tilde{W}_\alpha\right\vert_\xi^2(x)\, dx \le D_2$$
for all $R>0$ and all $y\in {\mathbb R}^3$. We deduce easily that~: for any $y\in {\mathbb R}^3$ and for any $R>0$, there exists $\frac{R}{2}\le r_\alpha\le R$ such that 
\begin{equation}\label{eq-claim-lba1-14}
\int_{\partial B_y\left(r_\alpha\right)} \left\vert {\mathcal L}_\xi \tilde{W}_\alpha\right\vert_\xi^2(x)\, dx \le 2D_2\hskip.1cm.
\end{equation}
We use now the Green representation formula, see Proposition \ref{vecteurGreen}, to write that 
\begincal
\left\vert {\mathcal L}_\xi \tilde{W}_\alpha\right\vert (y)&\le & D_3 \int_{B_y\left(r_\alpha\right)} \left\vert x-y\right\vert^{-2} \left\vert \overrightarrow{\Delta_\xi} \tilde{W}_\alpha(x)\right\vert\, dx + D_3 r_\alpha^{-2}\int_{\partial B_y\left(r_\alpha\right)}\left\vert {\mathcal L}_\xi \tilde{W}_\alpha\right\vert \, d\sigma \\
&\le &D_3 \int_{B_y\left(r_\alpha\right)} \left\vert x-y\right\vert^{-2} \left\vert \overrightarrow{\Delta_\xi} \tilde{W}_\alpha(x)\right\vert\, dx + \frac{2D_3\sqrt{8\pi D_2}}{R}\\
\fincal
thanks to (\ref{eq-claim-lba1-14}) and to the fact that $2r_\alpha\ge R$. Using (\ref{eqconv}), (\ref{eq-claim-lba1-10}), (\ref{eq-claim-lba1-12}), (\ref{eq-claim-lba1-13}) and the fact that $\tmua\to 0$ as $\alpha\to +\infty$, we get that 
$$ \int_{B_y\left(r_\alpha\right)} \left\vert x-y\right\vert^{-2} \left\vert \overrightarrow{\Delta_\xi} \tilde{W}_\alpha(x)\right\vert\, dx \to 0$$
as $\alpha\to +\infty$. Thus we obtain that 
$$\limsup_{\alpha\to +\infty} \left\vert {\mathcal L}_\xi \tilde{W}_\alpha\right\vert (y)\le \frac{2D_3\sqrt{8\pi D_2}}{R}\hskip.1cm.$$
Since this holds for all $R>0$, we have proved that 
\begin{equation}\label{eq-claim-lba1-15}
 {\mathcal L}_\xi \tilde{W}_\alpha\to 0\hbox{ in }L^\infty_{loc}\left({\mathbb R}^3\right)\hbox{ as }\alpha\to +\infty\hskip.1cm.
 \end{equation}
Then we have thanks to (\ref{eqconv}) and (\ref{eq-claim-lba1-13}) and since $\tilde{\mu}_\alpha\to 0$ as $\alpha\to +\infty$ that 
$$\mu_\alpha^2 \tilde{h}_\alpha\to 0,\, \tilde{f}_\alpha\to f\left(x_0\right)\hbox{ and }\tilde{a}_\alpha\to 0\hbox{ in }L^\infty_{loc}\left({\mathbb R}^3\right)\hbox{ as }\alpha\to +\infty\hskip.1cm,$$
where $x_\alpha\to x_0$ as $\alpha\to +\infty$ (up to a subsequence). By (\ref{eq-claim-lba1-10}) and (\ref{eq-claim-lba1-15}), there holds ${\ds \tilde{u}_\alpha(0)=\sqrt{2}+o(1)}$. Thanks to (\ref{eq-claim-lba1-10}) and (\ref{eq-claim-lba1-15}), the Harnack inequality of Proposition \ref{prop-HI} shows that $\left(\tilde{u}_\alpha\right)$ is uniformly positive in any compact set of ${\mathbb R}^3$ so that we can pass to the limit in equation (\ref{eq-claim-lba1-12}) to get that, after passing to a subsequence, 
\begin{equation}\label{eq-claim-lba1-16}
\tilde{u}_\alpha\to \tilde{u}\hbox{ in } C^{1,\eta}_{loc}\left({\mathbb R}^3\right)\hbox{ as }\alpha\to +\infty\hskip.1cm,
\end{equation}
where 
$$\Delta_\xi \tilde{u} = f_0\left(x_0\right) \tilde{u}^5\hbox{ in }{\mathbb R}^3$$
and 
$$0< \tilde{u}\le \tilde{u}\left(\tilde{y}_0\right)= \sqrt{2}$$
thanks to (\ref{eq-claim-lba1-10}) and (\ref{eq-claim-lba1-15}). Then, by the classification result of Caffarelli-Gidas-Spruck \cite{CaGiSp}, we have that 
$$\tilde{u} = \left(\frac{1+\frac{4f\left(x_0\right)}{3}\vert x-y_0\vert^2}{2}\right)^{-\frac{1}{2}}\hskip.1cm. $$
Since $\nabla \tilde{u}(0)=0$ thanks to (\ref{eq-claim-lba1-11}), we deduce that $\tilde{y}_0=0$ and thus that $\frac{\mua}{\tmua}\to 1$ as $\alpha\to +\infty$. The claim follows.\hfill $\diamondsuit$

\medskip We set in the following 
\begin{equation}\label{eq-lba-3}
B_\alpha(x) = \sqrt{2}\mu_\alpha^{\frac{1}{2}}\left(\mua^2 + \frac{4f_\alpha\left(x_\alpha\right)}{3} \left\vert x\right\vert^2\right)^{-\frac{1}{2}}
\end{equation}
satisfying ${\ds \Delta_\xi B_\alpha = f_\alpha\left(x_\alpha\right) B_\alpha^5}$ in ${\mathbb R}^3$. We also set for $x\in B_0(1)$~:
\begin{equation}\label{eq-lba-4}
\begin{array}{l}
{\ds \va(x)=u_\alpha\left[x_\alpha\right](x) U(x)\hskip.1cm,}\\
\,\\
{\ds \hWa(x)= U(x)^{-4} W_\alpha\left[x_\alpha\right](x)\hskip.1cm,}\\
\,\\
{\ds \hha(x)=\left(h_\alpha\left[x_\alpha\right](x)-\frac{3}{4}\right)U(x)^4\hskip.1cm,}\\
\,\\
{\ds \hfa(x)= f_\alpha\left[x_\alpha\right](x)\hskip.1cm,}\\
\,\\
{\ds\haa(x)= \left\vert U^2 \left(\pi_{x_\alpha}\right)_\star U_\alpha+U^6{\mathcal L}_\xi \hWa\right\vert_\xi^2(x) + U(x)^{12} b_\alpha\left[x_\alpha\right](x)\hskip.1cm,}\\
\,\\
{\ds\hXa(x)= U(x)^{-6}\left(\pi_{x_\alpha}\right)_\star X_\alpha(x)\hskip.1cm,}\\
\,\\
{\ds\hYa(x)= \left(\pi_{x_\alpha}\right)_\star Y_\alpha(x)\hskip.1cm.}
\end{array}
\end{equation}
so that equation (\ref{equationEalpha}) becomes thanks to (\ref{eqlapfctstereo}) and (\ref{eqlapvectstereo}) the following~:
\begin{equation}\label{eq-lba-5}
\left\{\begin{array}{l}
{\ds \Delta_\xi v_\alpha + \hha \va=\hfa \va^5+\haa \va^{-7}}\\
\,\\
{\ds \left(\overrightarrow{\Delta_\xi}\hWa\right)_i = -\frac{6x^j}{1+\vert x\vert^2}\left({\mathcal L}_\xi \hWa\right)_{ij} + \va^6\left(\hXa\right)_i+\left(\hYa\right)_i}
\end{array}\right.
\end{equation}
Note that claim \ref{claim-lba1} tells us that 
\begin{equation}\label{eq-lba-6}
\mua^{\frac{1}{2}} \va\left(\mu_\alpha x\right)\to U\left(\sqrt{\frac{4f\left(x_0\right)}{3}} x\right)\hbox{ in }C^1_{loc}\left({\mathbb R}^3\right)\hbox{ as }\alpha\to +\infty
\end{equation}
and that 
\begin{equation}\label{eq-lba7}
\mu_\alpha^3 \left\vert {\mathcal L}_\xi \hWa\right\vert_\xi\left(\mu_\alpha x\right)\to 0\hbox{ in }L^\infty_{loc}\left({\mathbb R}^3\right)\hbox{ as }\alpha\to +\infty\hskip.1cm.
\end{equation}
We set also 
\begin{equation}\label{eq-lba-8}
8\hpa=\tan\left(4\rho_\alpha\right)
\end{equation}
so that (\ref{eq-lba-1}) becomes 
\begin{equation}\label{eq-lba-9}
\vert x\vert^{3}\left(\va(x)^6 +\left\vert {\mathcal L}_\xi \hWa\right\vert_\xi(x)\right) \le C_2
\end{equation}
in $B_0\left(8\hpa\right)$ for some $C_2>0$ independent of $\alpha$ and $x$. 

\medskip Fix $\eps>0$. We define now $r_\alpha>0$ by 
\beq\label{eq-lba-10}
r_\alpha =\sup\left\{0\le r\le \hpa\hbox{ s.t. } \va< \left(1+\eps\right)B_\alpha\hbox{ in }B_{0}\left(r\right)\right\}
\eeq
where $\hpa$ is as in (\ref{eq-lba-8}). Thanks to (\ref{eq-lba-6}), it is clear that 
\beq\label{eq-lba-11}
\frac{r_\alpha}{\mu_\alpha}\to +\infty \hbox{ as }\alpha\to +\infty\hskip.1cm.
\eeq

\begin{claim}\label{claim-lba3}
There exists $C_3>0$ such that
$$v_\alpha\le C_3 B_\alpha$$
in $B_0\left(7r_\alpha\right)$. 
\end{claim}

\medskip {\bf Proof.} It is a direct consequence of Harnack's inequality, see Proposition \ref{prop-HI}. Indeed, let us set 
$$\tva= \left(\frac{r_\alpha}{2}\right)^{\frac{1}{2}}v_\alpha\left(\frac{1}{2}r_\alpha x\right)\hskip.1cm.$$
Then $\tva$ satisfies in $B_0\left(16\right)$ the equation 
$$\Delta_{\xi}\tva + \frac{1}{4}r_\alpha^2 \hha\left(\frac{1}{2}r_\alpha x\right)\tva  = \hfa\left(\frac{1}{2}r_\alpha x\right) \tva^5 + \left(\frac{r_\alpha}{2}\right)^{6}\frac{\haa\left(\frac{1}{2}r_\alpha x\right)}{\tva^{7}}\hskip.1cm.$$
Thanks to (\ref{eqconv}) and (\ref{eq-lba-9}), we know that there exists $D_1>0$ such that 
$$\vert x\vert^{\frac{1}{2}} \tva(x) + \vert x\vert^6 r_\alpha^6 \haa\left(\frac{1}{2}r_\alpha x\right)\le D_1$$
in $B_0(16)$. In particular, this implies that 
$$\tva \le D_1 \hbox{ in }B_0\left(16\right)\setminus B_0(1)\hskip.1cm.$$
We know moreover that 
$$\tva \le D_2 \mu_\alpha^{\frac{1}{2}} r_\alpha^{-\frac{1}{2}}\hbox{ in }B_0(2)\setminus B_0(1)$$
for some $D_2>0$ thanks to the definition of $r_\alpha$. Applying the Harnack inequality of Proposition \ref{prop-HI} on $B_{y}\left(1\right)$ for any $y\in B_0(3)\setminus B_0(2)$, we get that 
$$\tva\le D_3\mu_\alpha^{\frac{1}{2}} r_\alpha^{-\frac{1}{2}}\hbox{ in }B_0(3)\setminus B_0(2)$$
for some $D_3>0$ independent of $\alpha$. We can then repeat the argument on the annuli  $B_0(4)\setminus B_0(3)$, \dots, to finally get the existence of some $D_4>0$ independent of $\alpha$ such that 
$$\tva \le D_4 \mu_\alpha^{\frac{1}{2}} r_\alpha^{-\frac{1}{2}}\hbox{ in }B_0(14)\setminus B_0(2)\hskip.1cm.$$
This clearly leads to 
$$\va \le D_5 B_\alpha\hbox{ in }B_{0}\left(7r_\alpha\right)\setminus B_{0}\left(r_\alpha\right)$$
for some $D_5>0$ independent of $\alpha$. Since, by the definition of $r_\alpha$, we also have that 
$$v_\alpha\le \left(1+\eps\right) B_\alpha\hbox{ in }B_{0}\left(r_\alpha\right)\hskip.1cm,$$
this ends the proof of the claim.\hfill $\diamondsuit$

\medskip Note that, as a consequence of Claim \ref{claim1} and Claim \ref{claim-lba3}, we know that
\begin{equation}\label{eq-lba-11bis}
r_\alpha = O\left(\sqrt{\mua}\right)\hskip.1cm.
\end{equation}

\begin{claim}\label{claim-lba2}
There exists $\eta_\alpha\to 0$ such that 
$$\va\left(y_\alpha\right) \ge \left(1-\eta_\alpha\right) B_\alpha\left(y_\alpha\right)$$
for all sequences $\left(y_\alpha\right)$ of points in $B_0\left(8r_\alpha\right)$.
\end{claim}

\medskip {\bf Proof.} We let $G_\alpha$ be the Green function of $\Delta_h +h_\alpha$ on ${\mathbb S}^3$. Let $y_\alpha\in B_0\left(8r_\alpha\right)$. Since $\Delta_h u_\alpha +h_\alpha u_\alpha\ge 0$ thanks to (\ref{equationEalpha}) and the fact that $f_\alpha>0$ and $b_\alpha\ge 0$, we can write with the Green representation formula that 
$$v_\alpha\left(y_\alpha\right) \ge \mu_\alpha^{\frac{1}{2}} U\left(y_\alpha\right)\int_{B_0\left(8\frac{r_\alpha}{\mu_\alpha}\right)} \hfa\left(\mu_\alpha x \right) \left(\mu_\alpha^\frac{1}{2}v_\alpha\left(\mu_\alpha x\right)\right)^5 H_\alpha\left(y_\alpha,\mu_\alpha x\right)\, dx\hskip.1cm,$$
where 
$$H_\alpha\left(y_\alpha,y\right) = U\left(y\right) G_\alpha\left(\pi_{x_\alpha}^{-1}\left(y_\alpha\right),\pi_{x_\alpha}^{-1}\left(y\right)\right)\hskip.1cm.$$
We know thanks to (\ref{eqconv}) and to (\ref{eq-lba-11bis}) that 
$$\left\vert \left\vert y-\mu_\alpha x\right\vert H_\alpha\left(y_\alpha,\mu_\alpha x\right) - \frac{\sqrt{2}}{8\pi}\right\vert\to 0$$
uniformly for $x\in B_0\left(8\frac{r_\alpha}{\mu_\alpha}\right)$. Thus we have that 
\begincal
\frac{v_\alpha\left(y_\alpha\right)}{B_\alpha\left(y_\alpha\right)} &\ge &\bigl(1+o(1)\bigr) \frac{1}{8\pi}U\left(y_\alpha\right)\left(\mu_\alpha^2 +\frac{4f_\alpha\left(x_\alpha\right)}{3} \left\vert y_\alpha\right\vert^2\right)^{\frac{1}{2}}\\
&&\times\int_{B_0\left(8\frac{r_\alpha}{\mu_\alpha}\right)} \hfa\left(\mu_\alpha x \right) \left(\mu_\alpha^\frac{1}{2}v_\alpha\left(\mu_\alpha x\right)\right)^5 \left
\vert y_\alpha-\mu_\alpha x\right\vert^{-1}\, dx\hskip.1cm.
\fincal
Fatou's lemma together with simple computations lead then to the desired result thanks to Claim \ref{claim-lba1}.  \hfill $\diamondsuit$

\begin{claim}\label{claim-lba3bis}
There exists $C_4>0$ such that, for any $x\in B_0\left(7 r_\alpha\right)$, 
$$\int_{B_0\left(7r_\alpha\right)} \left\vert x-y\right\vert^{-1} \left\vert {\mathcal L}_\xi \hWa\right\vert_\xi^2(y) B_\alpha(y)^{-7}\, dy \le C_4 B_\alpha(x)\hskip.1cm.$$
\end{claim}

\medskip {\bf Proof. } We use the Green representation formula for $\Delta_\xi + \hha$ in $B_0\left(8r_\alpha\right)$ to write that there exists $D_1>0$, see \cite{RobWeb}, such that  
$$v_\alpha(x) \ge D_1\int_{B_0\left(7r_\alpha\right)} \frac{1}{\vert x-y\vert}\haa(y)v_\alpha(y)^{-7}\, dy$$
for all $x\in B_0\left(7r_\alpha\right)$ thanks to the fact that $v_\alpha\ge 0$. Here we also used equation (\ref{eq-lba-5}) and the fact that $\Delta_\xi v_\alpha+\hha v_\alpha\ge 0$. We write now that there exists $D_2>1$ such that 
$$\haa \ge D_2^{-1} \left\vert {\mathcal L}_\xi \hWa\right\vert^2_\xi - D_2$$
in $B_0\left(7r_\alpha\right)$. Using claim \ref{claim-lba3}, we get that 
\begincal
C_3 B_\alpha(x) &\ge& \frac{D_1}{D_2 C_3^7}\int_{B_0\left(7r_\alpha\right)} \frac{1}{\vert x-y\vert}\left\vert {\mathcal L}_\xi\hWa\right\vert_\xi^2(y) B_\alpha(y)^{-7}\, dy\\
&& - \frac{D_2 D_1}{C_3^7} \int_{B_0\left(7r_\alpha\right)} \frac{1}{\vert x-y\vert}B_\alpha(y)^{-7}\, dy\hskip.1cm.
\fincal
It remains to remark that there exists $D_3>0$ such that 
$$\int_{B_0\left(7r_\alpha\right)} \frac{1}{\vert x-y\vert}B_\alpha(y)^{-7}\, dy \le D_3 r_\alpha^9 \mu_\alpha^{-\frac{7}{2}}$$
and to note thanks to (\ref{eq-lba-11bis}) that 
$$r_\alpha^9 \mu_\alpha^{-\frac{7}{2}}\le D_4 B_\alpha(x)$$
for some $D_4>0$ for all $x\in B_0\left(7r_\alpha\right)$ to conclude. \hfill $\diamondsuit$

\medskip Let us define the $1$-form $V_\alpha$ in ${\mathbb R}^3$ by 
\begin{equation}\label{eq-lba-12}
V_\alpha(x)_i = \hXa\left(0\right)^j \int_{{\mathbb R}^3} B_\alpha(y)^6 {\mathcal H}_{ij}\left(x,y\right)\, dy\hskip.1cm,
\end{equation}
where 
$${\mathcal H}_{ij}\left(x,y\right)= \frac{1}{32\pi} \left(\frac{7\delta_{ij}}{\vert x-y\vert}+\frac{\left(x-y\right)_{i}\left(x-y\right)_j}{\vert x-y\vert^3}\right)\hskip.1cm.$$
 We have that 
\begin{equation}\label{eq-lba-13}
\overrightarrow{\Delta_\xi} V_\alpha = B_\alpha^6 \hXa(0)\hbox{ in }{\mathbb R}^3\hskip.1cm.
\end{equation}
We refer here to (\ref{convolformes}) in section \ref{fundsol}. We also let in the following 
\begin{equation}\label{eq-lba-13bis}
\eps_\alpha = \left\vert \hat{X}_\alpha(0)\right\vert
\end{equation}
and, if $\eps_\alpha\neq 0$, 
\begin{equation}\label{eq-lba-13bisbis}
\zeta= \lim_{\alpha\to +\infty}\frac{\hat{X}_\alpha(0)}{\eps_\alpha}
\end{equation}
which is a vector in ${\mathbb R}^3$ of norm $1$. Then direct computations give that 
\begin{equation}\label{eq-lba13ter}
\left\vert {\mathcal L}_\xi V_\alpha\right\vert_\xi (x)\le C_5 \frac{\eps_\alpha}{\mu_\alpha^2 +\vert x\vert^2}
\end{equation}
for all $x\in B_0\left(8r_\alpha\right)$ for some $C_5>0$ independent of $x$ and $\alpha$ and that, if $\eps_\alpha\neq 0$, 
\begin{equation}\label{eq-lba13-4}
\frac{r_\alpha^2}{\eps_\alpha} {\mathcal L}_\xi V_\alpha\left(r_\alpha x\right) \to  P_{ij}\hbox{ in } C^0_{loc}\left(B_0(4)\setminus\left\{0\right\}\right)\hbox{ as }\alpha\to +\infty\hskip.1cm,
\end{equation}
where 
\begin{equation}\label{eq-lba-13-5}
P_{ij}(x)=\frac{3\pi}{8}\left(\frac{3}{4f_0\left(x_0\right)}\right)^{\frac{3}{2}}  \left\vert x\right\vert^{-3} \left(\zeta_k x^k \left(\delta_{ij}-\frac{x_i x_j}{\vert x\vert^2}\right) - x_i \zeta_j-x_j \zeta_i  \right)\hskip.1cm.
\end{equation}

\begin{claim}\label{claim-lba4}
There exists $C_6>0$ such that 
$$\left\vert {\mathcal L}_\xi \hWa\right\vert_\xi (x)\le C_6 \frac{\mu_\alpha^2}{r_\alpha^3\left(\mu_\alpha^2+\vert x\vert^2\right)}$$
for all $x\in B_0\left(2r_\alpha\right)$ and all $\alpha$.
\end{claim}

\medskip {\bf Proof.} Let $z_\alpha\in B_0\left(4r_\alpha\right)$. Thanks to claim \ref{claim-lba3bis}, we know that 
$$\int_{B_0\left(6r_\alpha\right)\setminus B_0\left(5r_\alpha\right)} \left\vert z_\alpha - y\right\vert^{-1} \left\vert {\mathcal L}_\xi \hWa\right\vert_\xi^2(y)B_\alpha(y)^{-7}\, dy \le C_4B_\alpha\left(z_\alpha\right)\hskip.1cm.$$
This leads to the existence of some $s_\alpha\in \left(5r_\alpha,6r_\alpha\right)$ and of some $D_1>0$ independent of $\alpha$ such that 
\begin{equation}\label{eq-lba4-1}
\int_{\partial B_0\left(s_\alpha\right)} \left\vert {\mathcal L}_\xi \hWa\right\vert_\xi^2(y)\, d\sigma(y)\le D_1 \mu_\alpha^{\frac{7}{2}}r_\alpha^{-7} B_\alpha\left(z_\alpha\right)\hskip.1cm.
\end{equation} 
Thanks to the Green representation formula in $B_0\left(s_\alpha\right)$, see Proposition \ref{vecteurGreen}, there exists $D_2>0$ such that 
\begincal
\left\vert {\mathcal L}_\xi \left(\hWa-V_\alpha\right)\right\vert_\xi\left(z_\alpha\right) &\le & D_2 \int_{B_0\left(6r_\alpha\right)} \left\vert z_\alpha- y\right\vert^{-2} \left\vert {\overrightarrow \Delta}\left(\hWa-V_\alpha\right)(y)\right\vert_\xi\, dy\\
&& + D_2 \int_{\partial B_0\left(s_\alpha\right)} \left\vert z_\alpha-y\right\vert^{-2} \left\vert {\mathcal L}_\xi\left(\hWa-V_\alpha\right)\right\vert_\xi (y)\, d\sigma(y)\hskip.1cm.
\fincal
The boundary term can be estimated thanks to (\ref{eq-lba13ter}) and (\ref{eq-lba4-1}). We obtain that 
\begincal
&& \int_{\partial B_0\left(s_\alpha\right)} \left\vert z_\alpha-y\right\vert^{-2} \left\vert {\mathcal L}_\xi\left(\hWa-V_\alpha\right)\right\vert_\xi (y)\, d\sigma(y)\\
&&\quad \le D_3\left(\frac{\eps_\alpha}{r_\alpha^2} + \mu_\alpha^{\frac{7}{4}} r_\alpha^{-\frac{9}{2}} B_\alpha\left(z_\alpha\right)^{\frac{1}{2}}\right)
\fincal
for some constant $D_3>0$ independent of $\alpha$ so that we can write that 
\begincal
\left\vert {\mathcal L}_\xi \left(\hWa-V_\alpha\right)\right\vert_\xi\left(z_\alpha\right) &\le & D_2 \int_{B_0\left(6r_\alpha\right)} \left\vert z_\alpha- y\right\vert^{-2} \left\vert {\overrightarrow \Delta}\left(\hWa-V_\alpha\right)(y)\right\vert_\xi\, dy\\
&& + D_2  D_3\left(\frac{\eps_\alpha}{r_\alpha^2} + \mu_\alpha^{\frac{7}{4}} r_\alpha^{-\frac{9}{2}} B_\alpha\left(z_\alpha\right)^{\frac{1}{2}}\right)\hskip.1cm.
\fincal
Using equations (\ref{eq-lba-5}) and (\ref{eq-lba-13}), we have that 
\begincal
\left\vert {\overrightarrow \Delta}\left(\hWa-V_\alpha\right)(y)\right\vert_\xi&\le &\frac{6\vert y\vert}{1+\vert y\vert^2} \left\vert {\mathcal L}_\xi \hWa\right\vert_\xi (y) + \left\vert \hYa\right\vert_\xi(y)\\
&& +\left\vert v_\alpha(y)^6 \hXa(y)-B_\alpha(y)^6\hXa(0)\right\vert_\xi\hskip.1cm.
\fincal
Using (\ref{eqconv}), this leads with the previous inequality to the existence of some $D_4>0$ such that 
\begin{equation}\label{eq-lba4-2}
\left\vert {\mathcal L}_\xi \left(\hWa-V_\alpha\right)\right\vert_\xi\left(z_\alpha\right) \le D_4\left(I_\alpha^1+I_\alpha^2+I_\alpha^3+I_\alpha^4 +\frac{\eps_\alpha}{r_\alpha^2} + \frac{\mu_\alpha^{\frac{7}{4}}}{ r_\alpha^{\frac{9}{2}}} B_\alpha\left(z_\alpha\right)^{\frac{1}{2}}\right),
\end{equation}
where 
\begin{equation}\label{eq-lba4-2bis}
\begin{array}{l}
{\ds I_\alpha^1 = \int_{B_0\left(6r_\alpha\right)} \left\vert z_\alpha-y\right\vert^{-2} \vert y\vert \left\vert {\mathcal L}_\xi \hWa(y)\right\vert_\xi\, dy\hskip.1cm,}\\
{\ds I_\alpha^2 = \int_{B_0\left(6r_\alpha\right)}\left\vert z_\alpha-y\right\vert^{-2}\, dy\hskip.1cm,}\\
{\ds I_\alpha^3 = \eps_\alpha \int_{B_0\left(6r_\alpha\right)}\left\vert z_\alpha-y\right\vert^{-2}\left\vert v_\alpha(y)^6-B_\alpha(y)^6\right\vert \, dy\hskip.1cm,}\\
{\ds I_\alpha^4 = \int_{B_0\left(6r_\alpha\right)}\left\vert z_\alpha-y\right\vert^{-2} \vert y\vert v_\alpha(y)^6\, dy\hskip.1cm.}
\end{array}
\end{equation}
We clearly have that 
\begin{equation}\label{eq-lba4-3}
I_\alpha^2\le D_5 r_\alpha\hskip.1cm.
\end{equation}
Using Claim \ref{claim-lba3}, we can write by direct computations that 
\begin{equation}\label{eq-lba4-4}
I_\alpha^4 \le D_6 \frac{\mu_\alpha}{\mu_\alpha^2 +\left\vert z_\alpha\right\vert^2}\hskip.1cm.
\end{equation}
Thanks to Claim \ref{claim-lba1}, there exists $R_\alpha\to +\infty$ such that 
$$\mu_\alpha^{\frac{1}{2}}\left\Vert v_\alpha-B_\alpha\right\Vert_{L^\infty\left(B_0\left(R_\alpha\mu_\alpha\right)\right)}\to 0\hbox{ as }\alpha\to +\infty\hskip.1cm.$$
Then we write using also Claim \ref{claim-lba3} that 
\begincal
I_\alpha^3 &\le & \eps_\alpha \int_{B_0\left(R_\alpha\mu_\alpha\right)} \left\vert z_\alpha-y\right\vert^{-2} \left\vert v_\alpha(y)^6-B_\alpha(y)^6\right\vert\, dy\\
&& + O\left(\eps_\alpha \int_{B_0\left(6r_\alpha\right)\setminus B_0\left(R_\alpha\mu_\alpha\right)} \left\vert z_\alpha-y\right\vert^{-2} B_\alpha(y)^6\, dy\right)\\
&= & o\left(\eps_\alpha \mu_\alpha^{-\frac{1}{2}}\int_{B_0\left(R_\alpha\mu_\alpha\right)} \left\vert z_\alpha-y\right\vert^{-2}B_\alpha(y)^5\, dy\right)\\
&&+O\left(\eps_\alpha \int_{B_0\left(6r_\alpha\right)\setminus B_0\left(R_\alpha\mu_\alpha\right)} \left\vert z_\alpha-y\right\vert^{-2} B_\alpha(y)^6\, dy\right)\hskip.1cm.
\fincal
Simple computations lead then to 
\begin{equation}\label{eq-lba4-5}
I_\alpha^3 = o\left(\frac{\eps_\alpha}{\mu_\alpha^2+\left\vert z_\alpha\right\vert^2}\right)\hskip.1cm.
\end{equation}
In order to estimate $I_\alpha^1$, we use H\"older's inequalities with exponents $4$ and $\frac{4}{3}$ to write that 
\begincal
I_\alpha^1&\le & \left(\int_{B_0\left(6r_\alpha\right)} \left\vert z_\alpha-y\right\vert^{-1} B_\alpha(y)^{-7} \left\vert {\mathcal L}_\xi\hWa\right\vert_\xi^2(y)\, dy\right)^{\frac{1}{4}}\\
&& \left(\int_{B_0\left(6r_\alpha\right)} \left\vert z_\alpha-y\right\vert^{-\frac{7}{3}}\left\vert y\right\vert^{\frac{4}{3}} B_\alpha(y)^{\frac{7}{3}}\left\vert {\mathcal L}_\xi\hWa\right\vert_\xi^{\frac{2}{3}}(y)\, dy\right)^{\frac{3}{4}}\hskip.1cm.
\fincal
Using Claim \ref{claim-lba1} and (\ref{eq-lba-9}), we can write that 
$$\left\vert {\mathcal L}_\xi \hWa(y)\right\vert_\xi \le D_7\left(\mu_\alpha^2+\vert y\vert^2\right)^{-\frac{3}{2}}$$
for all $y\in B_0\left(6r_\alpha\right)$. Using this and Claim \ref{claim-lba3bis}, we get that 
$$I_\alpha^1 \le D_8 B_\alpha\left(z_\alpha\right)^{\frac{1}{4}} \left(\mu_\alpha^{\frac{7}{6}}\int_{B_0\left(6r_\alpha\right)} \left\vert y\right\vert^{\frac{4}{3}} \left\vert z_\alpha-y\right\vert^{-\frac{7}{3}}\left(\mu_\alpha^2 +\vert y\vert^2\right)^{-\frac{13}{6}}\, dy\right)^{\frac{3}{4}}\hskip.1cm.$$
Simple computations lead then to 
\begin{equation}\label{eq-lba4-6}
I_\alpha^1 \le D_9 \frac{\mu_\alpha}{\mu_\alpha^2+\left\vert z_\alpha\right\vert^2} \left(\ln \frac{2\left(\mu_\alpha^2+\left\vert z_\alpha\right\vert^2\right)}{\mu_\alpha^2}\right)^{\frac{3}{4}}\hskip.1cm.
\end{equation}
Coming back to (\ref{eq-lba4-2}) with (\ref{eq-lba4-2bis}), (\ref{eq-lba4-3}), (\ref{eq-lba4-4}), (\ref{eq-lba4-5}), (\ref{eq-lba4-6}) but also with (\ref{eq-lba-11bis}) and (\ref{eq-lba13ter}), we deduce that 
\begin{equation}\label{eq-lba4-7}
\begin{array}{l}
{\ds \left\vert {\mathcal L}_\xi \hWa\right\vert_\xi\left(z_\alpha\right)\le D_{10} \left(\frac{\mu_\alpha}{\mu_\alpha^2+\left\vert z_\alpha\right\vert^2} \left(\ln \frac{2\left(\mu_\alpha^2+\left\vert z_\alpha\right\vert^2\right)}{\mu_\alpha^2}\right)^{\frac{3}{4}}\right.}\\
{\ds \qquad\qquad\qquad\qquad \left.+ \frac{\eps_\alpha}{\mu_\alpha^2+\left\vert z_\alpha\right\vert^2}+ \mu_\alpha^{\frac{7}{4}} r_\alpha^{-\frac{9}{2}} B_\alpha\left(z_\alpha\right)^{\frac{1}{2}}\right)\hskip.1cm.}
\end{array}
\end{equation}
Thanks to this estimate on $\hWa$ in $B_0\left(4r_\alpha\right)$, we can sharpen the estimate on $I_\alpha^1$ for $z_\alpha\in B_0\left(2r_\alpha\right)$. Indeed, we can write that 
\begincal
I_\alpha^1&=& O\left(\int_{B_0\left(4r_\alpha\right)} \left\vert z_\alpha-y\right\vert^{-2} \left\vert y\right\vert \frac{\mu_\alpha}{\mu_\alpha^2+\left\vert y\right\vert^2} \left(\ln \frac{2\left(\mu_\alpha^2+\left\vert y\right\vert^2\right)}{\mu_\alpha^2}\right)^{\frac{3}{4}}\, dy\right)\\
&&+ O\left(\mu_\alpha^2 r_\alpha^{-\frac{9}{2}}\int_{B_0\left(4r_\alpha\right)} \left\vert z_\alpha-y\right\vert^{-2} \left\vert y\right\vert \left(\mu_\alpha^2+\vert y\vert^2\right)^{-\frac{1}{4}}\, dy\right)\\
&&+ O\left(\eps_\alpha \int_{B_0\left(4r_\alpha\right)} \left\vert z_\alpha-y\right\vert^{-2} \left\vert y\right\vert\left(\mu_\alpha^2+\vert y\vert^2\right)^{-1}\, dy\right)\\
&&+ O\left(r_\alpha^{-1} \int_{B_0\left(6r_\alpha\right)\setminus B_0\left(4r_\alpha\right)} \left\vert {\mathcal L}_\xi \hWa\right\vert_\xi(y)\, dy\right)\hskip.1cm.
\fincal
Direct computations lead to 
\begincal
I_\alpha^1&= & o\left(\frac{\mu_\alpha}{\mu_\alpha^2 +\left\vert z_\alpha\right\vert^2}\right)+O\left(\mu_\alpha^2 r_\alpha^{-3}\right)+O\left(\eps_\alpha \ln \frac{r_\alpha}{\mu_\alpha}\right)\\
&& + O\left(r_\alpha^{-1} \int_{B_0\left(6r_\alpha\right)\setminus B_0\left(4r_\alpha\right)} \left\vert {\mathcal L}_\xi \hWa\right\vert_\xi(y)\, dy\right)\hskip.1cm.
\fincal
In order to estimate the last term, we apply Claim \ref{claim-lba3bis} for some $\vert x\vert=r_\alpha$ to finally obtain 
\begin{equation}\label{eq-lba4-8}
I_\alpha^1 = o\left(\frac{\mu_\alpha}{\mu_\alpha^2 +\left\vert z_\alpha\right\vert^2}\right)+O\left(\eps_\alpha \ln \frac{r_\alpha}{\mu_\alpha}\right)+O\left(\mu_\alpha^2 r_\alpha^{-3}\right)\hskip.1cm.
\end{equation}
Coming back to (\ref{eq-lba4-2}) with (\ref{eq-lba4-2bis}), (\ref{eq-lba4-3}), (\ref{eq-lba4-4}), (\ref{eq-lba4-5}), (\ref{eq-lba4-8}) but also with (\ref{eq-lba-11bis}), we deduce that 
\begin{equation}\label{eq-lba4-9}
\begin{array}{l}
{\ds \left\vert {\mathcal L}_\xi\left(\hWa-V_\alpha\right)\right\vert_\xi\left(z_\alpha\right)}\\
{\ds \quad \le D_{11}\left(\frac{\eps_\alpha}{r_\alpha^2} + \mu_\alpha^{\frac{7}{4}} r_\alpha^{-\frac{9}{2}} B_\alpha\left(z_\alpha\right)^{\frac{1}{2}}+\frac{\mu_\alpha}{\mu_\alpha^2+\left\vert z_\alpha\right\vert^2}\right) + o\left(\frac{\eps_\alpha}{\mu_\alpha^2+\left\vert z_\alpha\right\vert^2}\right) \hskip.1cm.}
\end{array}\end{equation}
We claim now that 
\begin{equation}\label{eq-lba4-10}
\eps_\alpha = O\left(\mu_\alpha^2 r_\alpha^{-3}\right)\hskip.1cm.
\end{equation}
Indeed, we can write thanks to Claim \ref{claim-lba3bis} applied to some $\vert x\vert = 3r_\alpha$ that, for any $\delta>0$,
\begincal
&&\int_{B_0\left(2\delta r_\alpha\right)\setminus B_0\left(\delta r_\alpha\right)} B_\alpha(y)^{-7} \left\vert {\mathcal L}_\xi V_\alpha\right\vert_\xi^2(y)\, dy \\
&&\quad \le D_{12} \mu_\alpha^{\frac{1}{2}}+2\int_{B_0\left(2\delta r_\alpha\right)\setminus B_0\left(\delta r_\alpha\right)} B_\alpha(y)^{-7} \left\vert {\mathcal L}_\xi \left(\hWa-V_\alpha\right)\right\vert_\xi^2(y)\, dy
\fincal
where $D_{12}$ is of course independent of $\alpha$ and $\delta$. This leads with (\ref{eq-lba4-9}) to 
\begincal
&&r_\alpha^3 \int_{B_0\left(2\delta\right)\setminus B_0\left(\delta\right)} B_\alpha\left(r_\alpha y\right)^{-7} \left\vert {\mathcal L}_\xi V_\alpha\right\vert_\xi^2\left(r_\alpha y\right)\, dy \\
&&\quad \le D_{13} \left(\mu_\alpha^{\frac{1}{2}} + \frac{\eps_\alpha^2}{r_\alpha^4}\int_{B_0\left(2\delta r_\alpha\right)\setminus B_0\left(\delta r_\alpha\right)} B_\alpha(y)^{-7}\, dy\right.\\
&&\qquad \left. +\mu_\alpha^{\frac{7}{2}} r_\alpha^{-9}\int_{B_0\left(2\delta r_\alpha\right)\setminus B_0\left(\delta r_\alpha\right)} B_\alpha(y)^{-6}\, dy+\int_{B_0\left(2\delta r_\alpha\right)\setminus B_0\left(\delta r_\alpha\right)} B_\alpha(y)^{-3}\, dy\right)\\
&&\qquad + o\left(\frac{\eps_\alpha^2}{\mu_\alpha^2} \int_{B_0\left(2\delta r_\alpha\right)\setminus B_0\left(\delta r_\alpha\right)} B_\alpha(y)^{-3}\, dy\right)\hskip.1cm,
\fincal
where $D_{13}$ is independent of $\alpha$ and $\delta$. After simple computations, this gives using (\ref{eq-lba-11bis}) that 
$$r_\alpha^3 \int_{B_0\left(2\delta\right)\setminus B_0\left(\delta\right)} B_\alpha\left(r_\alpha y\right)^{-7} \left\vert {\mathcal L}_\xi V_\alpha\right\vert_\xi^2\left(r_\alpha y\right)\, dy 
 \le D_{14} \left(\mu_\alpha^{\frac{1}{2}}\left(1+\delta^9\right)+\delta^{10} \eps_\alpha^2 r_\alpha^6 \mu_\alpha^{-\frac{7}{2}}\right) $$
for some $D_{14}$ independent of $\alpha$ and $\delta$. Using now (\ref{eq-lba13-4}) and (\ref{eq-lba-13-5}), we can write that 
$$
r_\alpha^3 \int_{B_0\left(2\delta\right)\setminus B_0\left(\delta\right)} B_\alpha\left(r_\alpha y\right)^{-7} \left\vert {\mathcal L}_\xi V_\alpha\right\vert_\xi^2\left(r_\alpha y\right)\, dy 
\ge D_{15} \eps_\alpha^2 r_\alpha^6\mu_\alpha^{-\frac{7}{2}}\delta^6$$
for some $D_{15}$ independent of $\alpha$ and $\delta$. Up to choose $\delta>0$ small enough, we thus obtain that 
$$ \eps_\alpha^2 r_\alpha^6\mu_\alpha^{-\frac{7}{2}} = O\left(\mu_\alpha^{\frac{1}{2}}\right)$$
which leads to (\ref{eq-lba4-10}). 

Coming back to (\ref{eq-lba4-9}) with (\ref{eq-lba4-10}) but also with (\ref{eq-lba13ter}), we obtain the claim. \hfill $\diamondsuit$

\medskip We set now, for $x\in B_0(4)$, 
\begincal 
\cva(x)&=&\mua^{-\frac{1}{2}}r_\alpha v_\alpha\left(r_\alpha x\right)\hskip.1cm,\\
\cWa(x)&=& \hWa\left(r_\alpha x\right)\hskip.1cm,\\
\cha(x)&=& \hha\left(r_\alpha x\right)\hskip.1cm,\\
\cfa(x)&=& \hfa\left(r_\alpha x\right)\hskip.1cm,\\
\caa(x)&=& r_\alpha^2\haa\left(r_\alpha x\right)\hskip.1cm,\\
\cXa(x)&=& \hXa\left(r_\alpha x\right)\hskip.1cm,\\
\cYa(x)&=& \hYa\left(r_\alpha x\right)\hskip.1cm,\\
\cVa(x)&=& V_\alpha\left(r_\alpha x\right)\hskip.1cm,\\
\cZa(x)&=& Z_\alpha\left(r_\alpha x\right)\hskip.1cm.\\
\fincal
The first line of system (\ref{eq-lba-5}) becomes
\begin{equation}\label{eq-lba-26}
\Delta_\xi \cva + r_\alpha^2 \cha \cva = \left(\frac{\mu_\alpha}{r_\alpha}\right)^2 \cfa \cva^5 + \frac{r_\alpha^8}{\mu_\alpha^4}\frac{\caa}{\cva^7}\hbox{ in }B_0(4)\hskip.1cm.
\end{equation}
We can now use (\ref{eqconv}), (\ref{eq-lba-4}) and (\ref{eq-lba-11bis}) together with Claims \ref{claim-lba3} and \ref{claim-lba4} to write that 
\begin{equation}\label{eq-lba-32}
\frac{r_\alpha^8}{\mu_\alpha^4}\frac{\caa}{\cva^7} \le C_7 \left(\frac{\mu_\alpha^2}{r_\alpha^2}+\vert x\vert^2\right)^{\frac{3}{2}}
\end{equation}
for all $x\in B_0(2)$. 

\medskip Thanks to Claims \ref{claim-lba3} and \ref{claim-lba2}, we also know that 
\begin{equation}\label{eq-lba-27}
C_3 \mua^{-\frac{1}{2}}r_\alpha B_\alpha\left(r_\alpha x\right)\le \cva(x)\le C_3\mua^{-\frac{1}{2}}r_\alpha B_\alpha\left(r_\alpha x\right)\le C_3 \sqrt{\frac{3}{2 f_\alpha\left(x_\alpha\right)}} \vert x\vert^{-1}
\end{equation}
in $B_0(4)\setminus \left\{0\right\}$. We also have thanks to the definition (\ref{eq-lba-10}) of $r_\alpha$ that 
\begin{equation}\label{eq-lba-28}
\cva(x)< \left(1+\eps\right)\mua^{-\frac{1}{2}}r_\alpha B_\alpha\left(r_\alpha x\right)\hbox{ in }B_0(1)
\end{equation}
and that, if $r_\alpha<\hpa$, there exists $z_\alpha\in \partial B_0(1)$ such that 
\begin{equation}\label{eq-lba-29}
\cva\left(z_\alpha\right) = \left(1+\eps\right)\mua^{-\frac{1}{2}}r_\alpha B_\alpha\left(r_\alpha z_\alpha\right)\hskip.1cm.
\end{equation}
By standard elliptic theory, using (\ref{eq-lba-11}), (\ref{eq-lba-26}), (\ref{eq-lba-32}) and (\ref{eq-lba-27}) , we obtain that after passing to a subsequence, since $\cva$ is uniformly bounded from below in every compact subset of $B_0(2)\setminus\left\{0\right\}$ thanks to (\ref{eq-lba-27}), 
\begin{equation}\label{eq-lba-33}
\cva \to \cv \hbox{ in }C^1_{loc}\left(B_0(2)\setminus \left\{0\right\}\right)\hbox{ as }\alpha\to +\infty\hskip.1cm.
\end{equation}
Moreover, it is easily checked thanks to (\ref{eq-lba-26}) that 
\begin{equation}\label{eq-lba-35}
\cv(x)= \frac{\lambda}{\vert x\vert} + \beta(x)\hskip.1cm,
\end{equation}
where $\beta\in C^1\left(B_0(2)\right)$ is some super-harmonic function and ${\ds \lambda= \left(\frac{2f_0\left(x_0\right)}{3}\right)^{-\frac{1}{2}}}$. Moreover, using Claim \ref{claim-lba2}, it is also easily checked that 
\begin{equation}\label{eq-lba-36}
\beta(x)\ge 0 \hbox{ in }B_0(2)\hskip.1cm.
\end{equation}
Now we clearly have that 
\begin{equation}\label{eq-lba-37}
\beta>0 \hbox{ in }B_0(2) \hbox{ if } r_\alpha<\hpa
\end{equation}
thanks to (\ref{eq-lba-29}) and to the fact that $\beta$ is super-harmonic.

\begin{claim}\label{claim-lba5}
We have that $\beta(0)=0$ so that $r_\alpha=\hpa$.
\end{claim}

\medskip {\bf Proof.} - Let us apply the Poho\v{z}aev identity to $\cva$ in a ball $B_0\left(\delta\right)$. This reads as 
$$\int_{B_0\left(\delta\right)} \left(x^k\partial_k \cva+\frac{1}{2}\cva\right)\Delta_\xi \cva\, dx = \int_{\partial B_0\left(\delta\right)} \left(\frac{1}{2}\delta \left\vert \nabla\cva\right\vert_\xi^2 -\frac{1}{2} \cva\partial_\nu \cva-\delta\left(\partial_\nu \cva\right)^2\right)\, d\sigma\hskip.1cm.$$
Using (\ref{eq-lba-33}) and (\ref{eq-lba-35}), we obtain after simple computations that 
\begin{equation}\label{eq-lba5-1}
\lim_{\delta\to 0} \lim_{\alpha\to +\infty} \int_{B_0\left(\delta\right)} \left(x^k\partial_k \cva+\frac{1}{2}\cva\right)\Delta_\xi \cva\, dx = 2\pi \lambda \beta(0)\hskip.1cm.
\end{equation}
In order to estimate the left-hand side, we need some asymptotic of ${\ds x^k\partial_k \cva+\frac{1}{2}\cva}$ on the ball $B_0\left(\delta\right)$ for $\delta>0$ small enough. We have thanks to claim \ref{claim-lba1} that 
\begin{equation}\label{eq-lba5-2}
\frac{\mu_\alpha}{r_\alpha}\left(x^k \partial_k \cva +\frac{1}{2}\cva\right)\left(\frac{\mu_\alpha}{r_\alpha}\, \cdot\,\right)\to \frac{1}{2}\left(\frac{1}{2}+\frac{\vert x\vert^2}{\lambda^2}\right)^{-\frac{3}{2}}\left(\frac{1}{2}-\frac{\vert x\vert^2}{\lambda^2}\right)
\end{equation}
in $C^0_{loc}\left({\mathbb R}^3\right)$ as $\alpha\to +\infty$. We let now $\left(z_\alpha\right)$ be a sequence of points in $B_0\left(\delta\right)$ such that 
\begin{equation}\label{eq-lba5-3}
\frac{r_\alpha \vert z_\alpha\vert}{\mu_\alpha} \to +\infty\hbox{ as }\alpha\to +\infty
\end{equation}
and we write with the Green representation formula that 
$$\left(x^k\partial_k \cva+\frac{1}{2}\cva\right)\left(z_\alpha\right)=
\int_{B_0\left(1\right)} H^1\left(z_\alpha,x\right) \Delta_\xi \cva(x)\, dx+\int_{\partial B_0\left(1\right)} H^2(z_\alpha,x) \cva(x)d\sigma\hskip.1cm,$$
where 
$$H^1(z,x)=\frac{1}{8\pi}\left(\frac{\vert x\vert^2-\vert z\vert^2}{\left\vert x-z\right\vert^3}+\frac{\vert x\vert^2\vert z\vert^2-1}{\left\vert \vert z\vert x-\frac{z}{\vert z\vert}\right\vert^3}\right)$$
and 
$$H^2(z,x)=\frac{1}{8\pi \left\vert z-x\right\vert^5}\left(\vert z\vert^4-10\vert z\vert^2+1+4\left(1+\vert z\vert^2\right)\langle x,z\rangle\right)\hskip.1cm.$$
Thanks to (\ref{eq-lba-26}) and (\ref{eq-lba-33}), this leads to 
\begin{equation}\label{eq-lba5-4}
\begin{array}{rcl}
{\ds \left(x^k\partial_k \cva+\frac{1}{2}\cva\right)\left(z_\alpha\right)}&{\ds =}&{\ds \left(\frac{\mu_\alpha}{r_\alpha}\right)^2\int_{B_0\left(1\right)} H^1\left(z_\alpha,x\right) \cfa\cva^5\, dx}\\
\,&\, &{\ds -r_\alpha^2\int_{B_0\left(1\right)} H^1\left(z_\alpha,x\right) \cha \cva\, dx}\\
\, &\, &{\ds +\frac{r_\alpha^8}{\mu_\alpha^4}\int_{B_0\left(1\right)} H^1\left(z_\alpha,x\right) \caa \cva^{-7}\, dx}\\
\, &\, & {\ds +\int_{\partial B_0\left(1\right)} H^2\left(z_0,x\right) \cv(x)d\sigma+o(1)\hskip.1cm,}
\end{array}
\end{equation}
where ${\ds z_0=\lim_{\alpha\to +\infty}z_\alpha}$. Direct computations lead thanks to (\ref{eqconv}) and (\ref{eq-lba-27}) to 
\begin{equation}\label{eq-lba5-5}
\int_{B_0\left(1\right)} H^1\left(z_\alpha,x\right) \cha \cva\, dx = O\left(1\right)\hskip.1cm.
\end{equation}
Using (\ref{eq-lba-32}),  we obtain that 
\begin{equation}\label{eq-lba5-6}
\frac{r_\alpha^8}{\mu_\alpha^4}\int_{B_0\left(1\right)} H^1\left(z_\alpha,x\right) \caa \cva^{-7}\, dx \le D_2
\end{equation}
for some $D_2>0$ independent of $z_\alpha\in B_0\left(\delta\right)$. Let us take some $R_\alpha\to +\infty$ such that 
$$\frac{r_\alpha\vert z_\alpha\vert}{R_\alpha \mu_\alpha}\to +\infty\hbox{ as }\alpha\to +\infty$$
and such that 
$$\sup_{B_0\left(R_\alpha\frac{\mu_\alpha}{r_\alpha}\right)} \left\vert \frac{\cva}{\mu_\alpha^{-\frac{1}{2}}r_\alpha B_\alpha\left(r_\alpha x\right)}-1\right\vert \to 0\hbox{ as }\alpha\to +\infty\hskip.1cm.$$
Such a sequence $R_\alpha$ clearly exists thanks to claim \ref{claim-lba1} and to (\ref{eq-lba5-3}). Then we write thanks to (\ref{eqconv}) and to (\ref{eq-lba-27}) that 
$$ \left(\frac{\mu_\alpha}{r_\alpha}\right)^2\int_{B_0\left(1\right)\setminus B_0\left(R_\alpha\frac{\mu_\alpha}{r_\alpha}\right)} H^1\left(z_\alpha,x\right) \cfa\cva^5\, dx=o\left(\left\vert z_\alpha\right\vert^{-1}\right)\hskip.1cm.
$$
We can also write that 
$$\left(\frac{\mu_\alpha}{r_\alpha}\right)^2\int_{B_0\left(R_\alpha\frac{\mu_\alpha}{r_\alpha}\right)} H^1\left(z_\alpha,x\right) \cfa\cva^5\, dx =  -\frac{\lambda}{2\left\vert z_\alpha\right\vert}+o\left(\left\vert z_\alpha\right\vert^{-1}\right)$$
so that 
\begin{equation}\label{eq-lba5-7}
\left(\frac{\mu_\alpha}{r_\alpha}\right)^2\int_{B_0\left(1\right)} H^1\left(z_\alpha,x\right) \cfa\cva^5\, dx= -\frac{\lambda}{2\left\vert z_\alpha\right\vert}+o\left(\left\vert z_\alpha\right\vert^{-1}\right)\hskip.1cm.
\end{equation}
Coming back to (\ref{eq-lba5-4}) with (\ref{eq-lba5-5})-(\ref{eq-lba5-7}), we obtain that 
\begin{equation}\label{eq-lba5-8}
\left(x^k\partial_k \cva+\frac{1}{2}\cva\right)\left(z_\alpha\right)= -\frac{\lambda}{2\left\vert z_\alpha\right\vert}+o\left(\left\vert z_\alpha\right\vert^{-1}\right)+O(1)\hskip.1cm.
\end{equation}
And this holds for any sequence $\left(z_\alpha\right)$ in $B_0\left(\delta\right)$ satisfying (\ref{eq-lba5-3}). Thus, together with (\ref{eq-lba5-2}), we have obtained that for any sequence of points $\left(z_\alpha\right)$ in $B_0\left(\delta\right)$, 
\begin{equation}\label{eq-lba5-9}
\left(x^k\partial_k \cva+\frac{1}{2}\cva\right)\left(z_\alpha\right)= \frac{1}{2}\frac{r_\alpha}{\mu_\alpha} \frac{\frac{1}{2}-\frac{\vert z_\alpha\vert^2 r_\alpha^2}{\lambda^2 \mu_\alpha^2}}{\left(\frac{1}{2}+\frac{\vert z_\alpha\vert^2 r_\alpha^2}{\lambda^2 \mu_\alpha^2}\right)^{\frac{3}{2}}}+ O(1) + o\left(\left(\frac{\mu_\alpha}{r_\alpha}+\vert z_\alpha\vert\right)^{-1}\right)\hskip.1cm.
\end{equation}
In particular, there exists $\delta>0$ such that for $\alpha$ large, 
\begin{equation}\label{eq-lba5-10} 
x^k\partial_k \cva(x)+\frac{1}{2}\cva(x) \le 0 \hbox{ in }B_0\left(\delta\right)\setminus B_0\left(\lambda\frac{\mu_\alpha}{r_\alpha}\right)\hskip.1cm.
\end{equation}
Using now equation (\ref{eq-lba-26}), (\ref{eq-lba5-10}) and the fact that $\caa\ge 0$, we can write that 
\begincal
\int_{B_0\left(\delta\right)} \left(x^k\partial_k \cva+\frac{1}{2}\cva\right)\Delta_\xi \cva\, dx &\le &
\left(\frac{\mu_\alpha}{r_\alpha}\right)^2 \int_{B_0\left(\delta\right)} \left(x^k\partial_k \cva+\frac{1}{2}\cva\right)\cfa \cva^5\, dx\\
&&- r_\alpha^2\int_{B_0\left(\delta\right)} \left(x^k\partial_k \cva+\frac{1}{2}\cva\right)\cha \cva\, dx\\
&& + \frac{r_\alpha^8}{\mu_\alpha^4}\int_{B_0\left(\lambda\frac{\mu_\alpha}{r_\alpha}\right)} \left(x^k\partial_k \cva+\frac{1}{2}\cva\right)\caa \cva^{-7}\, dx\hskip.1cm.
\fincal
Let us write thanks to (\ref{eqconv}),  (\ref{eq-lba-11bis}), (\ref{eq-lba-27}) and (\ref{eq-lba5-9}) that
\begincal
&&\left(\frac{\mu_\alpha}{r_\alpha}\right)^2 \int_{B_0\left(\delta\right)} \left(x^k\partial_k \cva+\frac{1}{2}\cva\right)\cfa \cva^5\, dx \\
&&\quad =\cfa(0)\left(\frac{\mu_\alpha}{r_\alpha}\right)^2 \int_{B_0\left(\delta\right)} \left(x^k\partial_k \cva+\frac{1}{2}\cva\right) \cva^5\, dx\\
&&\qquad +\left(\frac{\mu_\alpha}{r_\alpha}\right)^2 \int_{B_0\left(\delta\right)} \left(x^k\partial_k \cva+\frac{1}{2}\cva\right)\left(\cfa-\cfa(0)\right) \cva^5\, dx\\
&&\quad = o(1)+ O\left(\left(\frac{\mu_\alpha}{r_\alpha}\right)^2 r_\alpha\int_{B_0\left(\delta\right)} \left\vert x^k\partial_k \cva+\frac{1}{2}\cva\right\vert \vert x\vert \cva^5\, dx\right)\\
&&\quad = o(1) \hskip.1cm.
\fincal
We also easily get that 
$$\int_{B_0\left(\delta\right)} \left(x^k\partial_k \cva+\frac{1}{2}\cva\right)\cha \cva\, dx = O(1)\hskip.1cm.$$
At last, using (\ref{eq-lba-27}) and (\ref{eq-lba5-9}), we obtain that 
$$\frac{r_\alpha^8}{\mu_\alpha^4}\int_{B_0\left(\lambda\frac{\mu_\alpha}{r_\alpha}\right)} \left(x^k\partial_k \cva+\frac{1}{2}\cva\right)\caa \cva^{-7}\, dx = O\left(\left(\frac{\mu_\alpha}{r_\alpha}\right)^5\right)\hskip.1cm.$$
Coming back to (\ref{eq-lba5-1}) with these last estimates, we obtain that $\beta(0)\le 0$. Thanks to (\ref{eq-lba-37}), this proves that $r_\alpha=\hpa$. This ends the proof of the claim. \hfill $\diamondsuit$

\medskip Now the results of this section clearly permit to prove Claim \ref{claim5}, combining the definition (\ref{eq-lba-10}) of $r_\alpha$ and Claims \ref{claim-lba2} and \ref{claim-lba5}.

\section{A Harnack inequality}\label{section-HI}

We prove in the following a Harnack inequality for solutions of Einstein-Lichnerowicz equation which was used throughout the paper~:

\begin{prop}\label{prop-HI}
Let $a,f,h$ be smooth functions on $\overline{B_0(2)}\subset {\mathbb R}^3$ and let $u\in C^2\left(B_0(2)\right)$ be a positive solution in $B_0(2)$ of 
$$\Delta_\xi u + h u = f u^5 + a u^{-7}\hskip.1cm.$$
We assume that 
\begin{itemize}
\item $u\in L^\infty\left(B_0(2)\right)$.
\item $f\ge 0$ and $a\ge 0$, $a\not\equiv 0$. 
\end{itemize}
Then there exists $C>0$ depending only on $\left\Vert h\right\Vert_{L^\infty\left(B_0(2)\right)}$, $\left\Vert a\right\Vert_{L^\infty\left(B_0(2)\right)}$, $\left\Vert f\right\Vert_{L^\infty\left(B_0(2)\right)}$  and $\left\Vert u\right\Vert_{L^\infty\left(B_0(2)\right)}$ such that
$$\sup_{B_0(1)} u\le C \inf_{B_0(1)} u\hskip.1cm.$$
\end{prop}

\medskip The proof follows the standard Nash-Moser iterative scheme. However, we have to deal with an additional difficulty due to the negative power nonlinearity of $u$ that makes the proof more involved. In particular, unlike the standard Harnack inequality, proposition \ref{prop-HI} is no longer an \emph{a priori} estimate and does not induce a control on the $L^\infty$-norm of $\nabla u$ in $B_0(1)$.

\medskip  {\bf Proof} - First of all, since $u$, $a$ and $f$ are nonnegative there holds
\[  \triangle_\xi u + hu \ge 0 \]
in $B_0(2)$. Hence Theorem $8.18$ in Gilbarg-Trudinger \cite{GilTru} applies and shows that for any $1 \le p < 3$, there exists $C_1(h,p)$ depending only on $\Vert h\Vert_{L^\infty(B_0(2))}$ and $p$ such that
\begin{equation} \label{Harnackminor}
\inf_{B_0(1)} u \ge C_1(h,p) \Vert u \Vert_{L^p(B_0(2))}.
\end{equation}

\noindent We now aim at proving that for any $p \ge 1$ there exist $C = C(a,h,f,u,p)$ such that
\[ \sup_{B_0(1)} u \le C \Vert u \Vert_{L^p(B_0(2))}\]
which together with \eqref{Harnackminor} will conclude the proof of the proposition. We adapt the steps of the proof of Theorem $4.1$ in Han-Lin \cite{HanLin}. Let $k \ge 8$ be given and $\eta \in C^\infty_c(B_0(2))$ be a smooth positive function with compact support in $B_0(2)$. Multiplying the equation satisfied by $u$ by $\eta^2 u^k$ and integrating yields 
\begincal
 \frac{4(k-1)}{(k+1)^2}  \int_{B_0(2)} \left | \nabla u^{\frac{k+1}{2}}\right |^2 \eta^2 dx &\le& \int_{B_0(2)} \left(  \eta^2 f u^{k+5}  + a \eta^2 u^{k-7} - \eta^2 h u^{k+1}\right)\, dx\\
&&+ \int_{B_0(2)} |\nabla \eta|^2 u^{k+1} \,dx \hskip.1cm.
\fincal
Let $0 < r < R \le 2$. Assume that $\eta$ is compactly supported in $B_0(R)$, is equal to $1$ in $B_0(r)$ and that it satisfies $| \eta | \le 1$ and $| \nabla \eta | \le \frac{2}{R-r}$ in $B_0(2)$. It is then easily seen that there exists $C_2>0$, depending only on $\Vert h \Vert_{L^\infty(B_0(2))}$, $\Vert a \Vert_{L^\infty(B_0(2))}$, $\Vert f \Vert_{L^\infty(B_0(2))}$ and $\Vert u \Vert_{L^\infty(B_0(2))}$, such that
$$
 \int_{B_0(2)} \left|  \eta^2 f u^{k+5}  + a \eta^2 u^{k-7} - \eta^2 h u^{k+1} \right| dx \le C_2 \left( \int_{B_0(R)}  u^{k+1} dx \right)^{\frac{k-7}{k+1}} 
$$
and that
$$
\int_{B_0(2)}  |\nabla \eta|^2 u^{k+1} dx \le C_2 (R-r)^{-2} \left( \int_{B_0(R)} u^{k+1} dx \right)^{\frac{k-7}{k+1}}.
$$
Independently, Sobolev's inequality shows that 
\[ \int_{B_0(2)} (\eta u^{\frac{k+1}{2}})^6 \le K \left( \int_{B_0(R)} \left | \nabla( \eta u^{\frac{k+1}{2}} )\right|^2 dx \right)^3\]
for some $K>0$ which leads to 
$$
\left( \int_{B_0(2)} (\eta^\frac{2}{k+1} u )^{3(k+1)} \right)^{\frac{1}{3}} \le C_3  \left( \int_{B_0(R)} |\nabla \eta |^2 u^{k+1}dx + \int_{B_0(2)} \eta^2 | \nabla u^{\frac{k+1}{2}}|^2 dx \right)
$$
for some positive $C_3$. We now let $\gamma = k +1 \ge 9$. Combining the above estimates, we obtain that 
\begin{equation} \label{improvedHolder}
\Vert u \Vert_{L^{3 \gamma}(B_0(r))} \le C_4^{\frac{1}{\gamma}} \left( \frac{\gamma}{(R-r)^2}\right)^{\frac{1}{\gamma}} \Vert u \Vert^{\frac{\gamma-8}{\gamma}}_{L^\gamma(B_0(R))}
\end{equation}
for some positive constant $C_4$ depending only on $\Vert h \Vert_{L^\infty(B_0(2))}$, $\Vert a \Vert_{L^\infty(B_0(2))}$, $\Vert f \Vert_{L^\infty(B_0(2))}$ and $\Vert u \Vert_{L^\infty(B_0(2))}$. We now pick some $0 < r < 2$ and define two sequences $\gamma_i$ and $r_i$ by $\gamma_i = 3^i \gamma$ and $r_0 = 2, r_{i+1} = r_i - (2-r)2^{-i-1}$. Inequality \eqref{improvedHolder} then gives that, for any $i \ge 0$:
\[ \Vert u \Vert_{L^{\gamma_{i+1}}(B_0(r_{i+1}))} \le C_4^{\frac{1}{3^i \gamma}} \left(  \frac{2 \cdot 6^i \gamma}{(2-r)^2}\right)^{\frac{1}{3^i \gamma}} \Vert u \Vert_{L^{\gamma_i}(B_0(r_i))} ^{1 - \frac{8}{3^i \gamma}} \]
and we thus obtain that there exists 
some constant $C_6>0$ depending on $\Vert h \Vert_{L^\infty(B_0(2))}$, $\Vert a \Vert_{L^\infty(B_0(2))}$, $\Vert f \Vert_{L^\infty(B_0(2))}$ and $\Vert u \Vert_{L^\infty(B_0(2))}$ but which does not depend on $r$ nor on $\gamma$ such that
\[\Vert u \Vert_{L^{\gamma_i}(B_0(r_i))} \le \frac{C_6}{(2-r)^{\frac{3}{\gamma}}} \Vert u \Vert_{L^{\gamma}(B_0(2))}^{\alpha} \]
for all $i \ge 0$, where $\alpha = \alpha(\gamma) = \prod_{k=0}^\infty  \left( 1 - \frac{8}{3^k \gamma} \right)$. Passing to the limit as $i \to \infty$ we thus obtain:
\begin{equation} \label{presqueHarnack}
\Vert u \Vert_{L^{\infty}(B_0(r))} \le \frac{C_6}{(2-r)^{\frac{3}{\gamma}}} \Vert u \Vert_{L^{\gamma}(B_0(2))}^{\alpha} .
\end{equation}
 
\noindent To conclude the proof of Proposition \ref{prop-HI}, we need to improve estimate \eqref{presqueHarnack}. Let $1 \le p < \gamma$. We bound $ (2-r)^{-3} \Vert u \Vert_{L^{\gamma}(B_0(2))}^{\alpha}$ using Young's inequality with exponents $\frac{\gamma}{\gamma - p \alpha}$ and $\frac{\gamma}{p \alpha}$ and combine with  \eqref{presqueHarnack} to obtain, for any $\eps > 0$:
\begin{equation} \label{Harnacketapefin}
\Vert u \Vert_{L^{\infty}(B_0(r))} \le C_6 \eps^{\frac{\gamma}{\gamma - p \alpha}}  \Vert u \Vert_{L^\infty(B_0(2))}^{\frac{(\gamma-p)\alpha}{\gamma - p \alpha}} + \frac{C_6}{(2-r)^{\frac{3 }{p \alpha}}} \frac{p}{\gamma} \eps^{- \frac{\gamma}{p \alpha}} \left( \int_{B_0(2)} u^p dx \right)^{\frac{1}{p}} .
\end{equation}
It is easily seen that $\alpha(\gamma) \to 1$ as $\gamma \to \infty$. Hence, choosing $\eps = (2 C_6)^{-1}$ one can then pick $\gamma$ large enough (depending on $a,h,f$ and $u$) so as to have, in \eqref{Harnacketapefin}:
\[ \Vert u \Vert_{L^{\infty}(B_0(r))}  \le \frac{2}{3} \Vert u \Vert_{L^\infty(B_0(2))} + C_7 (2-r)^{-\beta} \Vert u \Vert_{L^p(B_0(2))}\hskip.1cm, \]
where $C_7$ and $\beta > 0$ only depend on $\Vert h \Vert_{L^\infty(B_0(2))}$, $\Vert a \Vert_{L^\infty(B_0(2))}$, $\Vert f \Vert_{L^\infty(B_0(2))}$ and $\Vert u \Vert_{L^\infty(B_0(2))}$. The conclusion follows using Lemma $4.3$ in Han-Lin \cite{HanLin}. \hfill $\diamondsuit$

\section{Standard elliptic theory for the vectorial Laplacian on $\mathbb{S}^3$ }\label{stddelltheory}

Equation $\Dx X = 0$ historically appeared in the formulation of mathematical linear elasticity and is sometimes referred to as the Lam\'e system. We deal in this subsection with several properties of the operator $\Dh$ on $(\mathbb{S}^3, h)$. It is a differential operator between sections of the cotangent bundle $T^* \mathbb{S}^3$. If $X$ is a $1$-form in $\mathbb{S}^3$, $\overrightarrow{\triangle_h} X$ writes in coordinates as:
\[ \overrightarrow{\triangle_h} X_i = \nabla^j \nabla_j X_i + \nabla^j \nabla_i X^j - \frac{2}{3} \nabla_i ( \mathrm{div}_h X) .\]
If we write formally $\overrightarrow{\triangle_h} X(x) = L(x, \nabla) X$ then the principal symbol of the operator $\overrightarrow{\triangle_h} $ at some point $x \in \mathbb{S}^3$ and for some $\xi \in T_x \mathbb{S}^3$ is given by the determinant of the map $L(x,\xi)$ seen as a linear endomorphism of $T_x^* \Sp$. Thus there holds
\begin{equation} \label{unifelliptic}
 |L(x,\xi) | = \frac{4}{3} | \xi |_h^6 
 \end{equation} 
which shows that $\overrightarrow{\triangle_h} $ is uniformly elliptic in $\mathbb{S}^3$. It also satisfies the so-called strong ellipticity condition (also called Legendre-Hadamard condition) since for any $x \in \mathbb{S}^3$ and any $\eta \in T_x^*\Sp$:
\begin{equation} \label{strongelliptic}
(L(x,\xi) \eta)_i \eta^i = |\xi|_h^2 |\eta|_h^2 + \frac{1}{3} |\langle \xi, \eta \rangle|_h^2 \ge |\xi|_h^2 |\eta|_h^2\hskip.1cm.
\end{equation}
Since $\mathbb{S}^3$ is closed, integrating by parts, one gets that, for any $1$-forms $X$ and $Y$,
\begin{equation} \label{ippDeltah}
 \int_{\mathbb{S}^3}  \langle \Dh X, Y \rangle_h dv_h = \frac{1}{2} \int_{\Sp} \langle \mathcal{L}_h X, \mathcal{L}_h Y \rangle_h dv_h \hskip.1cm.
 \end{equation}
 In particular, \eqref{ippDeltah} shows that $\Dh$ is self-adjoint on $H^1(M)$ (we still denote the Sobolev space of $1$-forms by $H^1(M)$ since no ambiguity will occur) and that there holds in $\Sp$
 \begin{equation} \label{Killingh}
 \Dh X = 0 \Longleftrightarrow \mathcal{L}_h X = 0
 \end{equation}
 for any $1$-form $X$. Fields of $1$-forms in $\Sp$ satisfying $\mathcal{L}_h X = 0$ are called conformal Killing $1$-forms and by \eqref{ippDeltah} and standard Fredholm theory the set of those $1$-forms is finite dimensional. With \eqref{strongelliptic}, \eqref{ippDeltah} and \eqref{Killingh} standard results of elliptic theory for elliptic operators acting on vector bundles on closed manifolds apply, see for instance Theorem $27$, Appendix H in Besse \cite{Besse}, or Theorem $5.20$ in Giaquinta-Martinazzi \cite{GiaquintaMartinazzi}. In particular, for $1$-forms which are $L^2$-orthogonal to the subspace of conformal Killing forms, we have the following estimates~: 

 \begin{prop} \label{stdellipticclosed}
 For any $p > 1$, there exist constants $C_1 = C_1(h,p)$ and $C_2 = C_2(h,p)$ depending only on $h$ and $p$ such that for any $1$-form $X$ in $\Sp$:
 \begin{equation} \label{regellipstd}
  \Vert X \Vert_{W^{2,p}(\Sp)} \le C_1 \Vert \Dh X\Vert_{L^p(\Sp)} + C_2 \Vert X \Vert_{L^1(\Sp)}. 
  \end{equation}
  If, in addition, $X$ satisfies
 \begin{equation} \label{Killingortho}
  \int_{\Sp} \langle X, K \rangle_h dv_h  = 0
  \end{equation}
 for all conformal Killing $1$-form $K$, then, for any $p > 1$, we can choose $C_2 = 0$ in \eqref{regellipstd}. 
 \end{prop}

 \noindent It is in fact possible to fully describe the conformal Killing $1$-forms on $\mathbb{S}^3$ in terms of the conformal Killing forms in $\RR^3$.

 \begin{prop}\label{Killingsphere}
Let $P \in \Sp$. Then any conformal Killing $1$-form K in $\Sp$ is given by 
\begin{equation} \label{KillingS3}
Z = \pi_P^*(U^4 L),
\end{equation} 
where $\pi_P$ is the stereographic projection of pole $-P$, $U$ is as in \eqref{eqdefconffactor} and $L$ is some conformal Killing $1$-form in $\RR^3$, that is satisfying $\mathcal{L}_\xi L = 0$. 
 \end{prop}
 
 \begin{proof}
Let $Z$ satisfy $\mathcal{L}_h Z = 0$ in $\Sp$. Using \eqref{eqLgstereo}, there holds $\mathcal{L}_\xi ( U^{-4} (\pi_P)_* Z) = 0$. Conversely, let $L$ be a conformal Killing $1$-form in $\RR^3$. These forms are classified (see for instance Chapter $1$ of Schottenloher \cite{Schottenloher}) and span a $10$-dimensional vector space. In particular $L$ has the following expression: 
\begin{equation} \label{confclass}
L(x)_i = 2 \langle b,x \rangle_\xi x_i - |x|^2 b_i + \lambda x_i + c_i + (\Omega x)_i\hskip.1cm,
\end{equation}
where $\lambda \in \RR$, $b,c \in \RR^3$ and $\Omega$ is a skew-symmetric matrix. Note that \eqref{confclass} actually describes any conformal Killing $1$-form on any open subset of $\RR^3$. Let $P \in \Sp$ be some arbitrary point and define $Z = \pi_P^*(U^4 L)$. This defines a smooth conformal $1$-form in $\Sp \backslash \{ -P \}$. We now show that $Z$ is in $L^2(\Sp)$ and actually satisfies $\Dh Z = 0$ in $\Sp$ in a weak sense. Using the regularity theory as in proposition \ref{stdellipticclosed} this will show that $Z$ is a smooth conformal Killing $1$-form in $\Sp$ and \eqref{Killingh} will show that $Z$ is conformal. First, $Z \in L^2 (\Sp)$ since using \eqref{confclass} there holds, for any $y \in \Sp \backslash \{-P \}$:
\[ |Z(y)|_{h(y)} = 
U^2(\pi_P(y)) |L|_\xi(\pi_P(y)) \le  C\hskip.1cm, \]
where $C$ is some constant depending only on $n$ and $h$. Then, if $X$ is some smooth $1$-form in $\Sp$, integrating by parts we get for any $\eps > 0$:
\[\int_{\Sp \backslash B^h_{-P}(\eps) }\langle Z,\Dh X \rangle dv_h = \int_{\partial B^h_{-P}(\eps)} \langle \mathcal{L}_h X, Z \otimes \nu \rangle d \sigma_h = o(1) \]
so that letting $\eps \to 0$ shows that $\Dh Z = 0$ in a weak sense in $\Sp$.
 \end{proof}

 \section{Fundamental solution of the Lam\'e-type system in $\RR^3$.}\label{fundsol}

We define, for $1 \le i \le 3$, a $1$-form $H_i$ in $\RR^3 \backslash \{ 0 \}$ by:
\begin{equation} \label{Solfond}
H_{i}(y)_j = H_{ij}(y) = \frac{1}{32 \pi} \left( 7 \frac{\delta_{ij}}{|y|} + \frac{y_i y_j}{|y|^3} \right)
\end{equation}
for any $y \neq 0$. Note that the matrices $(H_{ij}(y))_{ij}$ thus defined are symmetric: for any $y \neq 0$,
\begin{equation} \label{solfondsym}
H_{ij}(y) = H_{ji}(y).
\end{equation}
Let $X$ be a field of $1$-forms in $\RR^3$. Integrating by parts and using Stoke's formula it is easily seen that for any $R > 0$ and for any $x \in B_0(R)$ there holds:
\begin{equation} \label{ippsolfond}
\begin{aligned}
X_i(x) = \int_{B_0(R)} H_{ij}(x-y) \Dx X(y)^j dx + \int_{\partial B_0(R)} \mathcal{L}_\xi X(y)^{kl} \nu_k(y) H_{il}(x-y) d\sigma \\
- \int_{\partial B_0(R)} \mathcal{L}_{\xi} \big( H_i(x - \cdot) \big)(y)_{kl} \nu(y)^k X(y)^l d \sigma.
\end{aligned}
\end{equation}
This means in a distributional sense that 
\[ \Dx \big( H_i (x - \cdot) \big) = \delta_x e_i\hskip.1cm, \]
where $e_i$ is the $i$-th vector of the canonical basis and there holds, for any $1$-form $Y$: $\langle \delta_x e_i, Y \rangle = Y_i(x)$. Equivalently, if we write $H(x,y) = (H_{ij}(x-y))_{1 \le i,j \le 3}$, we get that
\begin{equation} \label{solfondvect}
 \Dx H (x, \cdot) = \delta_x \mathrm{Id}\hskip.1cm,
 \end{equation}
where $\Dx$ is now seen as a matrix of differential operators acting on a distribution-valued matrix. Note that  the standard results of distribution theory easily extend to distribution-valued matrices, see for instance Schwartz \cite{Schwartz}. 

\noindent If now $X$ is some smooth field of $1$-forms in $\RR^3$ we let, for any $R >0$ and any $x \in B_0(R)$,
\begin{equation} \label{Zi}
 \begin{array}{l}
{\ds Z_i(x) =  \int_{\partial B_0(R)}\mathcal{L}_\xi X(y)^{kl} \nu_k(y) H_{il}(x-y) d\sigma }\\
{\ds \qquad- \int_{\partial B_0(R)} \mathcal{L}_{\xi} \big( H_i(x - \cdot) \big)(y)_{kl} \nu(y)^k X(y)^l d \sigma\hskip.1cm.} 
\end{array}
 \end{equation}
This defines a smooth $1$-form in $B_0(R)$ which satisfies
\begin{equation} \label{harmonicZi}
\Dx Z = 0 \quad \mathrm{ in } \quad B_0(R)
\end{equation}
due to \eqref{ippsolfond} and \eqref{solfondvect}. Similarly, using again \eqref{solfondvect}, if $Y \in L^1(\RR^3)$ is a smooth $1$-form then the $1$-form defined by 
\[ W_i(x) = \int_{\RR^3} H_{ij}(x-y) Y^j (y) dy = (H \star Y)_i(x) \]
satisfies in a weak sense
\begin{equation} \label{convolformes}
\Dx W_i (x) = Y_i(x) \hskip.1cm.
\end{equation}

\section{Green function with Neumann boundary conditions for Lam\'e-type systems in $\RR^3$.}

The system (\ref{equationEalpha}) we are interested in in this article is invariant up to adding to $W_\alpha$ some conformal Killing $1$-form in $\Sp$. We exploit this invariance all along the article by noting that the only relevant quantity to investigate is $\mathcal{L}_h W_\alpha$ and not the vector field $W_\alpha$ in itself. In particular we use several times a Green identity for $\Dx$ with Neumann boundary conditions that is proven in what follows. We let 
\begin{equation} \label{Killingspace}
K_R = \{ X \in H^1(B_0(R)), \mathcal{L}_\xi X = 0\}
 \end{equation}
be the Kernel subspace associated to the Neumann problem for $\Dx$ in $B_0(R)$. The orthogonal subspace of $K_R$ in $B_0(R)$ is the set of forms $Y \in H^1(B_0(R))$ such that for any $K \in K_R:$
\[ \int_{B_0(R)} \langle Y, K \rangle_\xi dx = 0 \quad \mathrm{holds}. \]
Elements of $K_R$ are infinitesimal generators of conformal transformations of $B_0(R)$ and are classified, see \eqref{confclass}. In particular $K_R$ is finite dimensional, $\dim K_R = 10$, and it is spanned by smooth vector fields. Let $(K_j)_{j=1 \dots 10}$ be an orthonormal basis of $K_0(R)$ for the $L^2$-scalar product, that is
\[ \int_{B_0(R)} \langle K_l, K_p \rangle_\xi dx =  \delta_{lp} .\]
The following proposition states the existence of Green $1$-forms satisfying Neumann boundary conditions:

\begin{prop} \label{vecteurGreen}
For any $1 \le i \le 3$ and any $R > 0$ there exists a unique $G_{i,R}$ defined in $B_0(R)\times B_0(R) \backslash D$, where $D = \{ (x,x), x \in B_0(R) \}$, such that $G_{i,R}(x, \cdot)$ is orthogonal to $K_R$ for any $x \in B_0(R)$ and such that for any smooth $1$-form $X$ in $\overline{B_0(R)}$:
\begin{equation} \label{ippNeumann}
\begin{array}{l}
{\ds \big( X - \pi_R(X) \big)_i(x)  = \int_{B_0(R)} G_{i,R}(x,y)_j \Dx X(y)^j dx} \\
{\ds \qquad\qquad\qquad\qquad\quad+ \int_{\partial B_0(R)} \mathcal{L}_\xi X(y)^{kl} \nu_k(y) G_{i,R}(x,y)_l d\sigma \hskip.1cm,}
\end{array}\end{equation}
where $\pi_R(X)$ is the orthogonal projection of $X$ on $K_R$ given by:
\begin{equation} \label{projkil}
 \pi_R(X) = \sum_{j=1}^{10} \left(  \int_{B_0(R)} \langle K_j, X\rangle dx \right) K_j .
 \end{equation}
Moreover $G_{i,R}$ is continuous and continuously differentiable in each variable in $B_0(R) \times B_0(R) \backslash D$. Furthermore, if $K$ denotes any compact set in $B_0(R)$ and if we let 
\begin{equation} \label{rapportdomaine}
 \delta = \frac{d(K, \partial B_0(R))}{R} > 0
 \end{equation}
there holds:
\begin{equation} \label{dersolNeumann}
 |x-y| |\nabla G_{i,R}(x,y)| + |G_{i,R}(x,y)| \le C(\delta) |x-y|^{-1} 
 \end{equation}
for any $ x\in B_0(K)$ and any $y \in B_0(R)$, whether the derivative in \eqref{dersolNeumann} is taken with respect to $x$ or $y$, and where $C(\delta)$ is a positive constant that only depends on $\delta$ as in \eqref{rapportdomaine} (in particular it does not depend on $x$).
\end{prop}

\begin{proof}

The proof of this proposition goes through a sequences of claims. The techniques used are strongly inspired from Robert \cite{RobWeb}.

\begin{claim} \label{claimexis}
Let $F$ and $G$ be smooth $1$-forms, in $B_0(R)$ and in $\partial B_0(R)$ respectively, satisfying:
\begin{equation} \label{compatibilite}
\int_{B_0(R)}  F_l K^l  d\xi + \int_{\partial B_0(R)} G_l K^l d\sigma= 0
\end{equation} 
for any $K \in K_R$, where $K_R$ is as in \eqref{Killingspace}.
Then there exists a unique smooth $1$-form $Z$ orthogonal to $K_R$ such that
\begin{equation} \label{eqmodele} \left \{ 
\begin{aligned}
& \Dx Z = F & \quad B_0(R) \\
& \nu^k \mathcal{L}_{\xi} Z_{kl} = G_l & \quad \partial B_0(R).\\
\end{aligned} \right. \end{equation}
\end{claim}

\begin{proof}
The existence and uniqueness of $Z$ is ensured by the Lax-Milgram theorem applied on the orthogonal complement of $K_R$ to the symmetric bilinear form
\[ B(X,Y) = \frac{1}{2} \int_{B_0(R)} \langle \mathcal{L}_\xi X, \mathcal{L}_\xi Y \rangle dx\]  
and to the linear  form:
\[ L(X) = \int_{B_0(R)} F_l X^l d\xi - \int_{\partial B_0(R)}  G_l X^l d\sigma.  \]
The coercivity of $B(X,X)$ on the orthogonal complement of $K_R$ follows from the definition of $K_R$ and is obtained via the direct method. We claim now that $Z$ is smooth in $B_0(R)$. This is a consequence of general elliptic regularity results up to the boundary for elliptic systems satisfying complementing boundary conditions, as stated in Agmon-Douglis-Nirenberg \cite{ADN2}. Due to \eqref{unifelliptic} and \eqref{strongelliptic} the problem \eqref{eqmodele} is complemented so that Theorem 10.5 in \cite{ADN2} applies and shows that $Z$ is smooth. 
\end{proof}
For any $1 \le i \le 3$ and any $x \in B_0(R)$ we let $U^R_{i,x}$ be the unique $1$-form in $H^1(\overline{B_0(R)})$, orthogonal to $K_R$, satisfying:
\begin{equation} \label{Uix}
\left \{ 
\begin{aligned}
& \Dx U^R_{i,x} = - \sum_{j=1}^{10} (K_j)_i(x) K_j & \quad B_0(R) \\
& \nu^k \mathcal{L}_{\xi} \big( U^R_{i,x}\big)_{kl} = - \nu^k \mathcal{L}_{\xi} \big( H_i(x- \cdot) \big)_{kl}  & \quad \partial B_0(R).\\
\end{aligned} \right.
\end{equation}
The existence and smoothness of $U^R_{i,x}$ is ensured by Claim \ref{claimexis}. Indeed, the compatibility condition \eqref{compatibilite} is satisfied by applying \eqref{ippsolfond} to any $K \in K_R$.

\medskip

\noindent We now let, for $x \neq y$:
\begin{equation} \label{solfondNeumann}
G_{i,R}(x,y) = H_i (x - y) + U^R_{i,x}(y) - \sum_{j=1}^{10} \left(  \int_{B_0(R)} \langle K_j, H_i(x- \cdot) \rangle dy \right) K_j(y).
\end{equation}
By construction, $G_{i,R}(x,\cdot)$ is a $1$-form defined in $B_0(R) \backslash \{ x\}$. It clearly belongs to $L^2(B_0(R))$, is orthogonal to $K_R$ and continuously differentiable in $B_0(R) \backslash \{ x\}$. Combining \eqref{ippsolfond} and \eqref{Uix} it is easily seen that \eqref{ippNeumann} holds. The next three claims aim at finishing the proof of the proposition. The first one is a uniqueness result. 

\begin{claim} \label{claimunicite}
Assume that for some $1 \le i \le 3$ and for some $x \in B_0(R)$ there exists a $1$-form $M_i$ in $L^1(B_0(R))$ such that for any $X \in C^2(\overline{B_0(R)})$ with $\mathcal{L}_\xi X_{kl} \nu^k = 0$ on $\partial B_0(R)$ there holds:
\begin{equation} \label{condunicite}
 \int_{B_0(R)} \langle M_i, \Dx X \rangle_\xi d\xi = \left( X - \pi(X)\right)_i(x) .
 \end{equation}
Then $G_{i,R}(x, \cdot) - M_i \in K_R$, where $K_R$ is as in \eqref{Killingspace}.

\begin{proof}
Let $F_i$ = $G_{i,R}(x, \cdot) - M_i$. Let $Y$ be a smooth $1$-form with compact support in $B_0(R)$. By Claim \ref{claimexis} there exists a smooth $1$-form $X$ in $B_0(R)$ such that $\Dx X = Y - \pi_R(Y)$ in $B_0(R)$ and $\mathcal{L}_\xi X_{kl} \nu^k = 0$ in $\partial B_0(R)$, where $\pi_R$ is as in \eqref{projkil}. Using \eqref{ippNeumann} and \eqref{condunicite} there holds:
\[ 0 =  \int_{B_0(R)} \langle F_i, \Dx X \rangle_\xi d\xi =   \int_{B_0(R)} \langle F_i, Y - \pi_R(Y) \rangle_\xi d\xi =  \int_{B_0(R)} \langle F_i - \pi_R(F_i), Y \rangle_\xi d\xi \]
by definition of $\pi_R$. Assume for a while that $F_i$ belongs to $L^p(B_0(R))$ for some $p > 1$. A density argument then shows that $F_i = \pi_R(F_i)$.

\noindent It thus remains to prove that $F_i \in L^p(B_0(R))$ for some $p > 1$. We only need to prove this for $M_i$. Let $p \in (1, 3)$ and define $q = \frac{p}{p-1}$. Let $Y$ be a smooth $1$-form compactly supported in $B_0(R)$. By Claim \ref{claimexis} there exists a smooth $1$-form $X$ in $B_0(R)$, orthogonal to $K_R$, such that $\Dx X = Y - \pi_R(Y)$ in $B_0(R)$ and $\mathcal{L}_\xi X_{kl} \nu^k = 0$ in $\partial B_0(R)$. Then the definition of $\pi_R$ and \eqref{condunicite} yield:
\begincal
  \int_{B_0(R)} \langle M_i - \pi_R(M_i), Y \rangle_\xi d\xi &=& \int_{B_0(R)} \langle M_i , Y - \pi_R(Y) \rangle_\xi d\xi \\
&=&  \int_{B_0(R)} \langle M_i ,  \Dx X \rangle_\xi d\xi \\
&=& X_i(x).
\fincal
Elliptic regularity results for complemented elliptic systems, as those stated in the proof of Claim \ref{claimexis}, show that there exists a constant $C$ only depending on $q$ such that $\Vert X \Vert_{W^{2,q}} \le C \Vert Y - \pi_R(Y) \Vert_{L^q}$ where we omit to say that these norms are taken on $B_0(R)$ for the sake of clarity. Since $q > \frac{3}{2}$ we thus obtain, using the Sobolev embedding of $W^{2,q}$ in $C^0(\overline{B_0(R)})$:
\[   \left | \int_{B_0(R)} \langle M_i - \pi_R(M_i), Y \rangle_\xi d\xi \right | \le C \Vert Y - \pi_R(Y) \Vert_{L^q} \le C \Vert Y \Vert_{L^q}.  \]
A density argument then shows that $M_i - \pi_R(M_i)$, and thus $M_i$, belongs to $L^p(B_0(R))$ for all $p \in (1, 3)$.
\end{proof}
\end{claim}

\noindent We now state some rescaling-invariance property of the Green $1$-forms $G_{i,R}$:
\begin{claim} \label{claimrescaling}
For any $R > 0$ there holds:
\begin{equation} \label{Greenrescaling}
G_{i,R}(x,y) = \frac{1}{R} G_{i,1} \left( \frac{x}{R}, \frac{y}{R} \right)
\end{equation} 
for any $x, y \in B_0(R)$.
\end{claim}

\begin{proof}
Let $Y$ be a smooth, compactly supported, $1$-form in $B_0(R)$. We define, in $B_0(1)$, $Y_R = Y(R \cdot)$, which is then compactly supported in $B_0(1)$ Let $x \in B_0(R)$ and $1 \le i \le 3$. Equation \eqref{ippNeumann} shows that
\begin{equation} \label{rescaleB1}
\left( Y_R - \pi_1(Y_R)\right)_i \left( \frac{x}{R} \right) = \int_{B_0(1)} \left \langle G_{i,1} \left( \frac{x}{R}, y \right), \Dx Y_R(y) \right \rangle_\xi d\xi.
\end{equation}
Since for any $y \in B_0(1)$ there holds $\Dx Y_R(y) = R^2 \Dx Y (Ry)$ we easily obtain: 
\[  \int_{B_0(1)} \left \langle G_{i,1} \left( \frac{x}{R}, y \right), \Dx Y_R(y) \right \rangle_\xi dy = \frac{1}{R} \int_{B_0(R)} \left \langle G_{i,1} \left( \frac{x}{R}, \frac{y}{R} \right) , \Dx Y (y) \right \rangle_\xi dy. \]
Let now $(L_j)_{1 \le j \le 10}$ be an orthonormal basis of $K_1$, where $K_1$ is as in \eqref{Killingspace}. Let $Z_j =  R^{- \frac{3}{2}} L_j \left( {\frac{x}{R}} \right)$ for any $1 \le j \le 10$ and any $x \in B_0(R)$. Then $(Z_j)_{1 \le j \le 10}$ is an orthonormal basis of $K_R$ since there holds
\[ \int_{B_0(R)} \langle Z_k, Z_l \rangle_\xi d\xi = \int_{B_0(R)} R^{-3} \langle L_k \left( \frac{y}{R}\right), L_l \left( \frac{y}{R}\right) \rangle_\xi dy = \int_{B_0(1)} \langle L_k, L_l \rangle_\xi d\xi = \delta_{kl}. \]
Hence one has, by definition of $\pi_R$:
\[ \begin{aligned}
\pi_R(Y) (x) & = \sum_{j = 1}^{10}   \left(  \int_{B_0(R)} \langle R^{- \frac{3}{2}} L_j \left( \frac{y}{R} \right), Y(y) \rangle dy \right) R^{- \frac{3}{2}} L_j \left(\frac{x}{R} \right)  \\
 & = \sum_{j = 1}^{10}   \left(  \int_{B_0(1)} \langle L_j (y), Y_R(y) \rangle dy \right) L_j \left( \frac{x}{R} \right)  = \pi_1(Y_R) \left( \frac{x}{R} \right). \\
 \end{aligned} \]
 In the end \eqref{rescaleB1} becomes:
 \[ \left( Y - \pi_R(Y) \right)_i(x) = \int_{B_0(R)} \left \langle \frac{1}{R}  G_{i,1} \left( \frac{x}{R}, \frac{y}{R} \right) , \Dx Y (y) \right \rangle_\xi dy.
 \]
Finally, $\frac{1}{R}  G_{i,1} \left( \frac{x}{R}, \frac{\cdot}{R} \right)$ is orthogonal to $K_R$: indeed for $1 \le j \le 10$ there holds:
 \begincal
\int_{B_0(R)} \left \langle \frac{1}{R}  G_{i,1} \left( \frac{x}{R}, \frac{y}{R}\right), R^{- \frac{3}{2}} L_j \left( \frac{y}{R} \right)  \right \rangle_\xi dy &=& R^{\frac{1}{2}} \int_{B_0(1)} \langle G_{i,1}( \frac{x}{R}, y), L_j(y) \rangle_\xi dy \\
&=& 0\hskip.1cm,
\fincal
where the last equality is true since $G_{i,1}( \frac{x}{R}, \cdot)$ is orthogonal to $K_1$ by definition. Using Claim \ref{claimunicite} we then obtain \eqref{Greenrescaling}. 
\end{proof}
\noindent The last ingredient of the proof is a symmetry property of $G_{i,R}$~:
\begin{claim} \label{symetrie}
For any $x,y \in B_0(R)$ there holds:
\begin{equation} \label{symGreen}
G_{i,R}(x,y)_j  =  G_{j,R}(y,x)_i - \pi_R \left( {}^t G_i(\cdot, x) \right)_j
 \end{equation}
where we have set:
\begin{equation} \label{greentranspo}
 {}^t G_i(\cdot, x)_j(y) = G_{j,R}(y,x)_i .
 \end{equation}
\end{claim}

\begin{proof}
Let $\Psi \in C^0(\overline{B_0(R)})$ be a $1$-form orthogonal to $K_R$. We define a $1$-form in $B_0(R)$ by
\begin{equation} \label{Hiform}
H_i(x) = \int_{B_0(R)} G_{j,R}(y,x)_i \Psi(y) ^j dy \hskip.1cm.
\end{equation}
By the explicit construction of $G_{j,R}$ in \eqref{solfondNeumann} it is easily seen that $H$ is continuous in $\overline{B_0(R)}$. Also, $H$ is orthogonal to the conformal Killing forms since by Fubini's theorem, for any $K \in K_R$, 
\begin{equation} \label{orthokil}
 \int_{B_0(R)} H_i(y) K^i (y) dy = \int_{B_0(R)} \Psi^j(z) \int_{B_0(R)} G_{j,R}(z,y)_i K^i(y) dy dz = 0 
 \end{equation}
since by construction $G_{j,R}(z, \cdot)$ is orthogonal to $K_R$ for any $z \in B_0(R)$. By Claim \ref{claimexis} and since $\Psi$ is orthogonal to $K_R$ we can let $F$ be the unique $C^1$ $1$-form in $B_0(R)$ orthogonal to $K_R$  satisfying $\Dx F = \Psi$ in $B_0(R)$ and $\mathcal{L}_\xi F_{kl} \nu^k = 0$ in $\partial B_0(R)$. Let also $\Phi$ be a smooth $1$-form such that $\mathcal{L}_\xi \Phi_{kl} \nu^k = 0$ on $\partial B_0(R)$. With Fubini's theorem, equation \eqref{ippNeumann} and using the properties of $\Psi$ and $\Phi$ there holds:
\[ \begin{aligned}
\int_{B_0(R)} H_i(y) \Dx \Phi^i(y) dy & = \int_{B_0(R)} \Psi^j (z) \int_{B_0(R)}   G_{j,R}(z,y)_i \Dx \Phi^i (y)  dy dz \\
& = \int_{B_0(R)} \Psi^j(z) \left( \Phi - \pi_R(\Phi)\right)_j (z) dz \\
& = \int_{B_0(R)}  \Psi^j(z) \Phi_j(z) dz \\
& = \int_{B_0(R)} \Phi_j(z) \Dx F^j (z) dz \\
& = \int_{B_0(R)} F_j(z) \Dx \Phi^j(z) dz \hskip.1cm,\\
\end{aligned} \]
where we integrated by parts to obtain the last inequality since the boundary terms vanish. In particular:
\[ \int_{B_0(R)} \left( H - F \right)_j(y) \Dx \Phi^j(y) dy = 0 \]
for any smooth $\Phi$ with $\mathcal{L}_\xi \Phi_{kl} \nu^k = 0$ on $\partial B_0(R)$. Note that by a density argument the above inequality remains true for $\Phi \in W^{2,p}(B_0(R))$ for any $p > 1$ and orthogonal to $K_R$. By construction $F$ is orthogonal to $K_R$ and thanks to \eqref{orthokil} so is $H$. By Claim \ref{claimexis} we can thus choose $\Phi$ to be the unique $1$-form orthogonal to $K_R$ satisfying $\Dx \Phi = F - H$ in $B_0(R)$ and $\mathcal{L}_\xi \Phi_{kl}\nu^k = 0$ in $\partial B_0(R)$ to obtain, with the above inequality, that $F = H$. Independently, using \eqref{ippNeumann} gives:
\[ F_i(x) = \int_{B_0(R)} G_{i,R}(x,y)_j \Psi^j(y) dy\]
so that
\begin{equation} \label{presquesym}
\int_{B_0(R)} \left( G_{j,R}(y,x)_i - G_{i,R}(x,y)_j \right) \Psi^j(y)dy = 0
\end{equation}
for any continuous, Killing-free $\Psi$. Let now $X$ be any smooth $1$-form in $B_0(R)$. Choose $\Psi = X - \pi_R(X)$. There holds
\begincal
&&\int_{B_0(R)} G_{j,R}(y,x)_i \pi_R(X)^j(y) dy \\ 
&&\quad = \sum_{p=1}^{10} \int_{B_0(R)} G_{j,R}(y,x)_i \left(  \int_{B_0(R)} \langle K_p(z), X(z)\rangle_\xi dz \right) K_p(y)^j dy   \\
&&\quad = \sum_{p=1}^{10} \int_{B_0(R)} X_l(z) \left(  \int_{B_0(R)} G_{j,R}(y,x)_i  K_p(y)^j dy \right)  K_p^l(z) dz \\
&&\quad = \sum_{p=1}^{10}  \int_{B_0(R)} X_l(z) \pi_R \left( {}^t G_i(\cdot, x) \right)^l(z) dz\hskip.1cm,
\fincal
where ${}^t G_i(\cdot, x)$ is as in \eqref{greentranspo}. Since $G_{i,R}(x, \cdot)$ has no conformal Killing part equation \eqref{presquesym} becomes:
\[ \int_{B_0(R)} \left( G_{j,R}(y,x)_i - G_{i,R}(x,y)_j -  \pi_R \left( {}^t G_i(\cdot, x) \right)_j(y) \right) X^j(y) dy = 0\]
for any smooth $1$-form $X$ and this concludes the proof of the claim.
\end{proof}

We are now able to end the proof of Proposition \ref{vecteurGreen}. Let $x \in B_0(1)$ and consider $U_{i,x}^1$ as defined in \eqref{Uix}. Since $\Dx$ is coercive on the orthogonal of $K_1$ we can use elliptic regularity results -- as those stated in the proof of Claim \ref{claimexis} -- to get that there exist positive constants $C_1, C_2$ that do not depend on $x$ such that
\begin{equation} \label{bornederUix}
 \Vert U_{i,x}^1 \Vert_{C^1(\overline{B_0(R)})} \le C_1 + C_2 \Vert \mathcal{L}_\xi H_i(x - \cdot), \nu \otimes \cdot \Vert_{C^1(\partial B_0(R))}\hskip.1cm,
 \end{equation}
where we have let $ ( \mathcal{L}_\xi H_i(x - \cdot), \nu \otimes \cdot)_l = \nu^k \mathcal{L}_{\xi} \big( H_i(x - \cdot) \big)_{kl}$. Let $K$ be some compact set in $B_0(1)$ and assume $x \in K$. It is easily seen by definition of $H_i$ as in \eqref{Solfond} that
\[ \Vert \mathcal{L}_\xi H_i(x - \cdot), \nu \otimes \cdot \Vert_{C^1(\partial B_0(R))} \le C_3 d \left(K, \partial B_0(1) \right)^{-2}\]
for some positive constant $C_3$ independent of $x$. By the definition of $G_{i,1}(x, \cdot)$ in \eqref{solfondNeumann} one therefore easily obtains that for $x \in K$ and $y \in B_0(1)$:
\begin{equation} \label{presquedersolNeumann}
  |x-y| |\nabla_y G_{i,1}(x,y)| + |G_{i,1}(x,y)| \le C(\delta) |x-y|^{-1}  \hskip.1cm,
  \end{equation}
where $\delta$ is as in \eqref{rapportdomaine}. This gives \eqref{dersolNeumann} when the derivative is taken with respect to $y$. The same estimate when the derivative is taken with respect to $x$ is obtained differentiating \eqref{symGreen} and combining with \eqref{presquedersolNeumann}. Finally, \eqref{dersolNeumann} for any positive $R$ is obtained combining \eqref{presquedersolNeumann} with Claim \ref{claimrescaling}.

\end{proof}

\bibliographystyle{amsplain}
\bibliography{biblio}

\providecommand{\bysame}{\leavevmode\hbox to3em{\hrulefill}\thinspace}
\providecommand{\MR}{\relax\ifhmode\unskip\space\fi MR }
\providecommand{\MRhref}[2]{%
  \href{http://www.ams.org/mathscinet-getitem?mr=#1}{#2}
}
\providecommand{\href}[2]{#2}
\begin{thebibliography}{10}

\bibitem{ADN2}
S.~Agmon, A.~Douglis, and L.~Nirenberg, \emph{Estimates near the boundary for
  solutions of elliptic partial differential equations satisfying general
  boundary conditions. {II}}, Comm. Pure Appl. Math. \textbf{17} (1964),
  35--92. \MR{0162050 (28 \#5252)}

\bibitem{Besse}
Arthur~L. Besse, \emph{Einstein manifolds}, Classics in Mathematics,
  Springer-Verlag, Berlin, 2008, Reprint of the 1987 edition. \MR{2371700
  (2008k:53084)}

\bibitem{Bre}
Simon Brendle, \emph{Blow-up phenomena for the {Y}amabe equation}, J. Amer.
  Math. Soc. \textbf{21} (2008), no.~4, 951--979. \MR{2425176 (2009m:53084)}

\bibitem{BreMa}
Simon Brendle and Fernando~C. Marques, \emph{Blow-up phenomena for the {Y}amabe
  equation. {II}}, J. Differential Geom. \textbf{81} (2009), no.~2, 225--250.
  \MR{2472174 (2010k:53050)}

\bibitem{CaGiSp}
Luis~A. Caffarelli, Basilis Gidas, and Joel Spruck, \emph{Asymptotic symmetry
  and local behavior of semilinear elliptic equations with critical {S}obolev
  growth}, Comm. Pure Appl. Math. \textbf{42} (1989), no.~3, 271--297.
  \MR{982351 (90c:35075)}

\bibitem{ChoGe}
Yvonne Choquet-Bruhat and Robert Geroch, \emph{Global aspects of the {C}auchy
  problem in general relativity}, Comm. Math. Phys. \textbf{14} (1969), no.~4,
  329--335.

\bibitem{ChristodoulouKlainerman}
Demetrios Christodoulou and Sergiu Klainerman, \emph{The global nonlinear
  stability of the {M}inkowski space}, Princeton Mathematical Series, vol.~41,
  Princeton University Press, Princeton, NJ, 1993. \MR{1316662 (95k:83006)}

\bibitem{DafermosRodnianski}
M.~Dafermos and I.~Rodnianski, \emph{Lectures on black holes and linear waves},
  arXiv:0811.0354.

\bibitem{DaGiHu}
Mattias Dahl, Romain Gicquaud, and Emmanuel Humbert, \emph{A limit equation
  associated to the solvability of the vacuum {E}instein constraint equations
  by using the conformal method}, Duke Math. J. \textbf{161} (2012), no.~14,
  2669--2697. \MR{2993137}

\bibitem{DruetENSAIOS}
Olivier Druet, \emph{La notion de stabilit\'e pour des \'equations aux
  d\'eriv\'ees partielles elliptiques}, Ensaios Matem\'aticos [Mathematical
  Surveys], vol.~19, Sociedade Brasileira de Matem\'atica, Rio de Janeiro,
  2010. \MR{2815304}

\bibitem{DruHeb}
Olivier Druet and Emmanuel Hebey, \emph{Stability and instability for
  {E}instein-scalar field {L}ichnerowicz equations on compact {R}iemannian
  manifolds}, Math. Z. \textbf{263} (2009), no.~1, 33--67. \MR{2529487
  (2010h:58028)}

\bibitem{ChoBru}
Y.~Four{\`e}s-Bruhat, \emph{Th\'eor\`eme d'existence pour certains syst\`emes
  d'\'equations aux d\'eriv\'ees partielles non lin\'eaires}, Acta Math.
  \textbf{88} (1952), 141--225. \MR{0053338 (14,756g)}

\bibitem{GiaquintaMartinazzi}
Mariano Giaquinta and Luca Martinazzi, \emph{An introduction to the regularity
  theory for elliptic systems, harmonic maps and minimal graphs}, second ed.,
  Appunti. Scuola Normale Superiore di Pisa (Nuova Serie) [Lecture Notes.
  Scuola Normale Superiore di Pisa (New Series)], vol.~11, Edizioni della
  Normale, Pisa, 2012. \MR{3099262}

\bibitem{GilTru}
David Gilbarg and Neil~S. Trudinger, \emph{Elliptic partial differential
  equations of second order}, Classics in Mathematics, Springer-Verlag, Berlin,
  2001, Reprint of the 1998 edition. \MR{1814364 (2001k:35004)}

\bibitem{HanLin}
Qing Han and Fanghua Lin, \emph{Elliptic partial differential equations},
  second ed., Courant Lecture Notes in Mathematics, vol.~1, Courant Institute
  of Mathematical Sciences, New York, 2011. \MR{2777537}

\bibitem{HebeyZLAM}
Emmanuel Hebey, \emph{Compactness and stability for nonlinear elliptic
  equations}, Zurich lectures in advanced mathematics (ZLAM), European
  Mathematical Society, to appear.

\bibitem{HePaPo}
Emmanuel Hebey, Frank Pacard, and Daniel Pollack, \emph{A variational analysis
  of {E}instein-scalar field {L}ichnerowicz equations on compact {R}iemannian
  manifolds}, Comm. Math. Phys. \textbf{278} (2008), no.~1, 117--132.
  \MR{2367200 (2009c:58041)}

\bibitem{HeVer}
Emmanuel Hebey and Giona Veronelli, \emph{The {L}ichnerowicz equation in the
  closed case of the {E}instein-{M}axwell theory}, Trans. Amer. Math. Soc.
  \textbf{366} (2014), no.~3, 1179--1193. \MR{3145727}

\bibitem{HolstMeier}
M.~Holst and C.~Meier, \emph{Non uniqueness of solutions to the conformal
  formulation}, arXiv:1210.2156.

\bibitem{Ise}
James Isenberg, \emph{Constant mean curvature solutions of the {E}instein
  constraint equations on closed manifolds}, Classical Quantum Gravity
  \textbf{12} (1995), no.~9, 2249--2274. \MR{1353772 (97a:83013)}

\bibitem{KhuMaSc}
M.~A. Khuri, F.~C. Marques, and R.~M. Schoen, \emph{A compactness theorem for
  the {Y}amabe problem}, J. Differential Geom. \textbf{81} (2009), no.~1,
  143--196. \MR{2477893 (2010e:53065)}

\bibitem{Lich}
Andr{\'e} Lichnerowicz, \emph{L'int\'egration des \'equations de la gravitation
  relativiste et le probl\`eme des {$n$} corps}, J. Math. Pures Appl. (9)
  \textbf{23} (1944), 37--63. \MR{0014298 (7,266d)}

\bibitem{Premoselli2}
Bruno Premoselli, \emph{Effective multiplicity for the {E}instein-scalar field
  {L}ichnerowicz equation}, Accepted for publication in Calculus of Variations
  and Partial Differential Equations.

\bibitem{Premoselli1}
Bruno Premoselli, \emph{The {E}instein-{S}calar {F}ield {C}onstraint {S}ystem
  in the {P}ositive {C}ase}, Comm. Math. Phys. \textbf{326} (2014), no.~2,
  543--557. \MR{3165467}

\bibitem{Rendall}
Alan~D. Rendall, \emph{Accelerated cosmological expansion due to a scalar field
  whose potential has a positive lower bound}, Classical and Quantum Gravity
  \textbf{21} (2004), no.~9, 24--45.

\bibitem{RobWeb}
Fr{\'e}d{\'e}ric Robert, \emph{Existence et asymptotiques optimales des
  fonctions de green des op{\'e}rateurs elliptiques d'ordre deux},
  http://www.iecn.u-nancy.fr/~frobert/ConstrucGreen.pdf.

\bibitem{Schottenloher}
M.~Schottenloher, \emph{A mathematical introduction to conformal field theory},
  second ed., Lecture Notes in Physics, vol. 759, Springer-Verlag, Berlin,
  2008. \MR{2492295 (2011a:81219)}

\bibitem{Schwartz}
Laurent Schwartz, \emph{Th{\'e}orie des distributions {\`a} valeurs
  vectorielles. {I}}, Ann. Inst. Fourier, Grenoble \textbf{7} (1957), 1--141.
  \MR{0107812 (21 \#6534)}

\end{thebibliography}

\end{document}